\documentclass[11pt]{article}
\usepackage{fullpage}
\usepackage{graphicx}
\usepackage{amssymb}
\usepackage{amsmath}
\usepackage{amsthm}

\usepackage{thmtools}
\usepackage{thm-restate}

\usepackage{xcolor}

\usepackage{wasysym}

\usepackage{hyperref}
\usepackage[capitalise]{cleveref}

\declaretheorem[name=Theorem,numberwithin=section]{theorem}

\newtheorem{lemma}{Lemma}

\newtheorem{conjecture}{Conjecture}

\crefname{conjecture}{Conjecture}{Conjectures}

\theoremstyle{definition}
\newtheorem{definition}{Definition}
\newtheorem{remark}{Remark}
             
\newcommand{\ignore}[1]{}
             
\newcommand{\hcm}[1][1]{\hspace*{#1 cm}}
\newcommand{\rb}[2]{\raisebox{#1 mm}[0mm][0mm]{#2}}
\newcommand{\istrut}[2][0]{\rule[- #1 mm]{0mm}{#1 mm}\rule{0mm}{#2 mm}}
\newcommand{\zero}[1]{\makebox[0mm][l]{$#1$}}

\newcommand{\E}{{\mathbb E\/}}

\newcommand{\floor}[1]{\left\lfloor #1 \right\rfloor}

\newcommand{\fr}[2]{\mbox{$\frac{#1}{#2}$}}

\newcommand{\bydef}{\stackrel{\operatorname{def}}{=}}
\newcommand{\poly}{\operatorname{poly}}
\newcommand{\MOD}{\operatorname{mod}}

\newcommand{\ul}[1]{\underline{#1}}

\newcommand{\Erdos}{Erd\H{o}s}
\newcommand{\Renyi}{R\'{e}nyi}

\newcommand{\Furedi}{F\"{u}redi}

\newcommand{\Kyncl}{Kyn\v{c}l}

\newcommand{\Korandi}{Kor\'{a}ndi}
\newcommand{\Kovari}{K\H{o}v\'{a}ri}
\newcommand{\Sos}{S\'{o}s}
\newcommand{\Turan}{Tur\'{a}n}
\newcommand{\Gyori}{Gy\H{o}ri}

\newcommand{\Ex}{\operatorname{Ex}}
\newcommand{\ExT}{\Ex_{\mathit{Tur}}}
\newcommand{\ExH}{\Ex_{<}}
\newcommand{\bu}{\ensuremath\bullet}

\renewcommand{\S}{\mathcal{S}}

\newcommand{\Oneeven}[1]{\mathbf{1}_0({#1})}
\newcommand{\Oneodd}[1]{\mathbf{1}_1({#1})}
\newcommand{\ieven}{i^{\operatorname{even}}}
\newcommand{\iodd}{i^{\operatorname{odd}}}

\newcommand{\Row}{\mathbf{r}}
\newcommand{\Col}{\mathbf{c}}

\newcommand{\sig}{\operatorname{sig}}
\newcommand{\slab}{\operatorname{slab}}

\newcommand{\first}{\mathsf{first}}
\newcommand{\last}{\mathsf{last}}

\title{A Refutation of the Pach-Tardos Conjecture for 0--1 Matrices\thanks{S. Pettie is supported by NSF Grant CCF-2221980.  G. Tardos is supported by the National Research, Development and Innovation Office projects K-132696 and SNN-135643 and by the ERC Advanced Grants ``ERMiD'' and ``GeoScape.''}}

\author{Seth Pettie\\ University of Michigan \and G\'abor Tardos\\ \Renyi{} Institute}

\date{}

\begin{document}

\maketitle

\begin{abstract}
The theory of forbidden 0--1 matrices generalizes \Turan-style (bipartite) subgraph avoidance, Davenport-Schinzel theory, and Zarankiewicz-type problems,
and has been influential in many areas, such as 
discrete and computational geometry,
the analysis of self-adjusting data structures, 
and the development of the graph parameter \emph{twin width}.

\medskip 

The foremost open problems in this area is to resolve the \emph{Pach-Tardos conjecture} from 2005,
which states that if a forbidden pattern $P\in\{0,1\}^{k\times l}$ is the bipartite incidence matrix of an \emph{acyclic} graph (forest), then $\Ex(P,n) = O(n\log^{C_P} n)$, where $C_P$ is a constant depending only on $P$.  This conjecture has been confirmed on many small patterns, specifically all $P$ 
with weight at most 5, and all but two with weight 6.

\medskip 

The main result of this paper is a clean refutation 
of the Pach-Tardos conjecture.  
Specifically, we prove that 
$\Ex(S_0,n),\Ex(S_1,n) \geq n2^{\Omega(\sqrt{\log n})}$,
where $S_0,S_1$ are the 
outstanding weight-6 patterns.

\begin{align*}
    S_0 &= \left(\begin{array}{cccc}
        \bu & & \bu & \\
        \bu &&& \bu\\
        &\bu&&\bu
    \end{array}\right),
&
    S_1 &= \left(\begin{array}{cccc}
        \bu &&& \bu\\
        \bu & & \bu & \\
        &\bu&&\bu
    \end{array}\right),
&
P_t &= \left(\begin{array}{ccccccc}
\bu&\zero{\hcm[-.1]\rb{4}{$\overbrace{\hcm[3.35]}^{\text{$t+1$ alternating 1s}}$}}\bu &&\bu&&&\bu\\
\bu&&\bu&&\rb{2.4}{$\cdots$}&\bu&
\end{array}\right).
\end{align*}

\medskip 

We also prove sharp bounds on the entire class of \emph{alternating}
patterns $(P_t)$, specifically that for every $t\geq 2$, $\Ex(P_t,n)=\Theta(n(\log n/\log\log n)^t)$.
This is the first proof of an asymptotically sharp bound 
that is $\omega(n\log n)$.
\end{abstract}

\section{Introduction}

The extremal theory of 
pattern-avoiding 0--1 matrices
kicked off in the late 1980s
when Mitchell~\cite{Mitchell92}, 
Pach and Sharir~\cite{PachS91},
and \Furedi~\cite{Furedi90} applied 
forbidden matrix arguments to problems 
in discrete and computational geometry.
In the early days this theory was characterized~\cite{Mitchell92,FurediH92} as a 
\emph{two dimensional} generalization of Davenport-Schinzel theory~\cite{AgarwalSharir95,Pettie-DS-JACM}.
It can also be characterized as a generalization of
\Turan{} theory~\cite{Turan41,FurediS13} from unordered 
bipartite graphs to \emph{ordered} bipartite graphs.
\Furedi{} and Hajnal~\cite{FurediH92} 
(see also Bienstock and \Gyori~\cite{BienstockG91}) 
began the daunting project of \emph{classifying} all
forbidden patterns by their extremal function,
a project to which many researchers have made important contributions over the years~\cite{Klazar92,KV94,Tardos05,PachTardos06,Keszegh09,Cibulka09,Fulek09,Geneson09,Pettie-FH11,Pettie-GenDS11,Fox13,ParkS13,Pettie-DS-JACM,Pettie15-SIDMA,CibulkaK17,GenesonT17,WellmanP18,GyoriKMTTV18,KorandiTTW19,Geneson19,FurediKMV20,GenesonT20,MethukuT22,GerbnerMNPTV23,KucheriyaT23,KucheriyaT23b,JanzerJMM24,ChalermsookPY24,PettieT24}.
Before proceeding let us define the terms.

\subsection{Forbidden Patterns, 0--1 Matrices, Extremal Functions}

A matrix $A\in\{0,1\}^{n\times m}$ \emph{contains} a pattern $P\in\{0,1\}^{k\times l}$,
written $P\prec A$, if it is possible to transform $A$ into $P$ by removing rows and columns from $A$, and flipping 1s to 0s.  If $P\nprec A$ we say $A$ is \emph{$P$-free}.
Define the \emph{extremal functions} as follows.
\begin{align*}
    \Ex(P,n,m) &= \max\{\|A\|_1 \mid A\in \{0,1\}^{n\times m} \mbox{ and } P\nprec A\},\\
    \Ex(P,n) &= \Ex(P,n,n),
\end{align*}
where $\|A\|_1$, the \emph{weight} of the matrix $A$, 
is the number of 1s in it.  $P\in\{0,1\}^{k\times l}$ 
can be regarded as the adjacency matrix of a bipartite graph with $k+l$ vertices,
where the parts of the partition (rows and columns) are implicitly \emph{ordered}.
Define $G(P)$ to be the unordered bipartite graph corresponding to $P$.
\Turan's extremal function $\ExT(H,n)$ is defined to be the maximum number of edges 
in a simple $n$-vertex graph not containing $H$ as a subgraph.  

\subsection{The Classification of Patterns} 

We have a crude classification of forbidden 
subgraphs according to the asymptotic behavior 
of their \Turan-extremal functions.
\begin{itemize}
    \item If $H$ is non-bipartite, then $\ExT(H,n)=\Theta(n^2)$.\footnote{\Erdos{}, Stone, and Simonovits~\cite{ErdosS46,ErdosS66}, generalizing \Turan's theorem~\cite{Turan41} bounded it more precisely as $\ExT(H,n) = (1-1/r+o(1)){n\choose 2}$ if $H$ has chromatic number $r+1$.}
    \item If $H$ is bipartite and contains a cycle, 
    then $\ExT(H,n) = \Omega(n^{1+c_1})$ and $O(n^{1+c_2})$ for $0<c_1<c_2<1$.\footnote{\Erdos{} and Simonovits conjectured that
$\Ex(H,n) = \Theta(n^{1+\alpha})$ for some rational $\alpha\in \mathbb{Q}$; see~\cite{FurediS13}.}
    \item If $H$ is acyclic (a forest) then $\ExT(H,n)=\Theta(n)$.
\end{itemize}

Is there a similarly clean asymptotic classification 
for forbidden patterns in 0--1 matrices?  
In a very influential paper, \Furedi{} and Hajnal~\cite{FurediH92}
observed that (trivially) $\Ex(P,n) =\Omega(\ExT(G(P),2n))$
and that there were several examples when 
$\Ex(P,n) = \omega(\ExT(G(P),2n))$, e.g.,\footnote{Following convention, we write patterns using 
bullets for 1s and blanks for 0s.}
\begin{align*}
P_1 &= \left(\begin{array}{ccc}
\bu&\bu\\
\bu&&\bu\end{array}\right),
&
Q_3  &= \left(\begin{array}{cccc}
\bu&&\bu\\
&\bu&&\bu
\end{array}\right).
\end{align*}
Both are \emph{acyclic matrices}, so $\ExT(G(P_1),n)=\ExT(G(Q_3),n)=O(n)$, but 
$\Ex(P_1,n)=\Theta(n\log n)$ \cite{BienstockG91,Furedi90,FurediH92,Tardos05}
and $\Ex(Q_3,n) = \Theta(n\alpha(n))$~\cite{FurediH92,HS86}.
The pattern $P_1$ arises in an analysis of the 
Bentley-Ottman line-sweeping algoithm~\cite{PachS91}, 
bounding unit distances in convex $n$-gons~\cite{Furedi90},
and bounding the total length of 
path compressions on arbitrary trees~\cite{Pettie10a}.
The pattern $Q_3$ corresponds to order-$3$ ($ababa$-free) 
Davenport-Schinzel sequences, which have applications
to lower envelopes~\cite{WS88,AgarwalSharir95} and corollaries of the dynamic optimality conjecture~\cite{ChalermsookGJAPY23,Pettie-Deque-08,ChalermsookG0MS15,ChalermsookPY24}.

\subsection{The \Furedi-Hajnal and Pach-Tardos Conjectures}

\Furedi{} and Hajnal made three 
conjectures concerning the relationship between 
$\Ex$ and $\ExT$.

\begin{conjecture}[\Furedi{} and Hajnal~\cite{FurediH92}]\label{conj:FH-perm}
If $P$ is a permutation matrix (equivalently, $G(P)$ is a matching),
then $\Ex(P,n) = O(\ExT(G(P),n)) = O(n)$.
\end{conjecture}

\begin{conjecture}[\cite{FurediH92}]\label{conj:FH-logn}
For \underline{\emph{any}} $P$, $\Ex(P,n)=O(\ExT(G(P),n)\cdot \log n)$.
\end{conjecture}

Perhaps doubting the validity of \cref{conj:FH-logn} \emph{in general}, they asked whether it held at least for acyclic patterns.

\begin{conjecture}[\cite{FurediH92}]\label{conj:FH-acyclic}
For any acyclic $P$, $\Ex(P,n)=O(\ExT(G(P),n)\cdot \log n)=O(n\log n)$.
\end{conjecture}

In 2004, Marcus and Tardos~\cite{MarcusT04} 
proved \cref{conj:FH-perm}, which also proved 
the \emph{Stanley-Wilf conjecture}, 
via a prior reduction of Klazar~\cite{Klazar00}.
This result inspired a line of research that led 
to the definition of the graph parameter \emph{twin width}~\cite{GuillemotM14,BonnetGKTW21,BonnetKTW22}.
Although the leading constant in $\Ex(P,n)=O(n)$ for a $k$-permutation $P$ depends only on $k$, it is \emph{exponentially} larger than the corresponding leading constant of $\ExT(G(P),n)=O(n)$; see Fox~\cite{Fox13} 
and Cibulka and \Kyncl~\cite{CibulkaK17}.

In 2005 Pach and Tardos~\cite{PachTardos06} refuted \cref{conj:FH-logn}.  They provided a matrix with 
$\|A\|_1 = \Theta(n^{4/3})$ that for each $k$, avoids a certain pattern $D_{2k}$ for which $G(D_{2k})=C_{2k}$ is a $2k$-cycle.  
Since $\ExT(C_{2k},n)=O(n^{1+1/k})$~\cite{BondyS74}, this proved that the gap between $\Ex(P,n)$ and $\ExT(G(P),n)$ can be as large as $n^{1/3-\epsilon}$ for any $\epsilon>0$.
This result had no direct effect on \cref{conj:FH-acyclic}, but cast some doubt on its validity.  Before \cref{conj:FH-acyclic} was refuted they stated a more plausible version of it.

\begin{conjecture}[Pach and Tardos~\cite{PachTardos06}]\label{conj:PachTardos}
Let $P$ be an acyclic 0--1 pattern.
\begin{description}
    \item[Weak Version.] $\Ex(P,n) = O(n\log^{C_P} n)$, for \emph{some} constant $C_P$.
    \item[Strong Version.] $\Ex(P,n) = O(n\log^{\|P\|_1-3} n)$.
\end{description}
\end{conjecture}

The rationale for the {\bfseries Strong Version}  is that all acyclic $P$ 
with weight 3 are known to be linear~\cite{FurediH92}, and in 
\emph{some circumstances}, adding a row/column containing a single 1 only increases the extremal function by a $\log n$ factor.  In particular, 
Pach and Tardos~\cite{PachTardos06} 
proved the following three reductions for eliminating weight-1 columns.
(In the diagrams, there are no 
constraints on the order of the rows.)

\begin{lemma}[Pach and Tardos~\cite{PachTardos06}]\label{lem:PachTardos-reductions}
Suppose $P$ is obtained from $P'$ (marked by boxes) 
by adding weight-1 columns in the following configurations.\footnote{Formally: 
(A) The last column of $P$ has one 1. 
(B) Column $j$ of $P$ has one 1.  There are 
rows $i_0,i_1$ such that 
$P(i_0,j)=P(i_0,j+1)=P(i_1,j-1)=P(i_1,j+1)=1$.
(C) Columns $j$ and $j+1$ of $P$ have one 1 each.
There are rows $i_0,i_1,i_2$ such that 
$P(i_0,j-1)=P(i_0,j)=P(i_1,j+1)=P(i_1,j+2)=P(i_2,j-1)=P(i_2,j+2)=1$.

Since $P$ and its transpose have the same extremal function, 
reductions (A), (B), and (C) can also be applied to 
remove weight-1 rows. Strictly speaking, part (C) is implied by part (B).}

\bigskip 
\begin{tabular}{ccc}
    $\displaystyle P = \left(\,\begin{array}{|ccc|c}
    \cline{1-3}&&&\\
    &P'&&\bu\\
    &&&\\\cline{1-3}\end{array}\,\right)$
    &
    $\displaystyle P = \left(\,\begin{array}{|ccc|c|ccc|}
    \cline{1-3}\cline{5-7}&&\bu&&\bu&&\\
    &P'&&&&&\\
    &&&\bu&\bu&&\\\cline{1-3}\cline{5-7}\end{array}\,\right)$
    &
    $\displaystyle P = \left(\,\begin{array}{|ccc|cc|ccc|}
    \cline{1-3}\cline{6-8}&&\bu&\bu&&&&\\
    &&\bu&&&\bu&&\\
    $P'$&&&&\bu&\bu&&\\\cline{1-3}\cline{6-8}\end{array}\,\right)$\\
    \istrut[3]{6}\hcm[.7]\emph{(A)} & \hcm\emph{(B)} & \hcm\emph{(C)}
\end{tabular}

Then $\Ex(P,n)$ can be expressed in terms of $\Ex(P',n)$ as follows.
\begin{description}
    \item[(A)] $\Ex(P,n) = O(\Ex(P',n)\log n)$.
    \item[(B)] $\Ex(P,n) = O(\Ex(P',n)\log n)$.
    \item[(C)] $\Ex(P,n) = O(\Ex(P',n)\log^2 n)$.
\end{description}
\end{lemma}

The Pach-Tardos reductions (\cref{lem:PachTardos-reductions})
are sufficient to prove \cref{conj:PachTardos} on all patterns with weight at most 5 and most of weight 6.  For example,
consider the class $(P_t)$ of ``alternating'' patterns
and $R_0,R_1,R_2$.
\begin{align*}
    P_t &= \left(\begin{array}{ccccccc}
\bu&\zero{\hcm[-.2]\rb{4}{$\overbrace{\hcm[3.6]}^{\text{$t+1$ alternating 1s}}$}}\bu &&\bu&&&\bu\\
\bu&&\bu&&\rb{2.4}{$\cdots$}&\bu&
\end{array}\right),
&
    R_0 &= \left(\begin{array}{ccc}
    \bu&\bu\\
    &&\bu\\
    \bu&&\bu\end{array}\right),\\
&\\
    R_1 &= \left(\begin{array}{cccc}
    \bu&\bu\\
    &&\bu&\bu\\
    \bu&&&\bu\end{array}\right),
&
    R_2 &= \left(\begin{array}{cccc}
    \bu&\bu\\
    \bu&&&\bu\\
    &&\bu&\bu\\
\end{array}\right).
\intertext{A $t$-fold application of \cref{lem:PachTardos-reductions}(A) implies $\Ex(P_t,n) = O(n\log^t n)$, 
(B) implies $\Ex(R_0,n) = O(n\log^2 n)$, and
(C) implies $\Ex(R_1,n) = O(n\log^3 n)$ and 
$\Ex(R_2,n) = O(n\log^2 n)$.
However, there are two weight-6 patterns up to rotation/reflection that the Pach-Tardos reductions cannot simplify, namely $S_0$ and $S_1$.}
    S_0 &= \left(\begin{array}{cccc}
        \bu & & \bu & \\
        \bu &&& \bu\\
        &\bu&&\bu
    \end{array}\right),
&
    S_1 &= \left(\begin{array}{cccc}
        \bu &&& \bu\\
        \bu & & \bu & \\
        &\bu&&\bu
    \end{array}\right).
\end{align*}

\subsection{Acyclic Patterns and 
the Status of Conjectures \ref{conj:FH-acyclic}
and \ref{conj:PachTardos}}

In 2010 Pettie~\cite{Pettie-FH11} 
refuted \Furedi{} and Hajnal's \cref{conj:FH-acyclic} 
by exhibiting an acyclic pattern $X$ for which $\Ex(X,n)=\Omega(n\log n\log\log n)$.
\[
X = \left(\begin{array}{ccccc}
	& \bu	&		& \bu	& \bu\\
	&		& \bu 	&		& \bu\\
	& \bu\\
\bu &		&		&		& \bu\end{array}\right).
\]
Park and Shi~\cite{ParkS13} generalized this construction to a class $(X_m)$ of acyclic patterns for which $\Ex(X_m,n)=\Omega(n\log n\log\log n\log\log\log n\cdots\log^{(m)} n)$.  
These results~\cite{Pettie-FH11,ParkS13} 
did not cast any doubt on the Pach-Tardos conjecture (\cref{conj:PachTardos}),
and even left open the possibility that the \Furedi-Hajnal conjecture (\cref{conj:FH-acyclic}) was still \emph{morally true}, 
e.g., if $\Ex(P,n) = O(n\log n\poly(\log\log n))$ for all acyclic $P$.

Essentially no progress has been made on expanding 
Pach and Tardos's repertoire of weight-1-column 
reduction rules (\cref{lem:PachTardos-reductions}) 
in order to put more acyclic matrices in the $n\poly(\log n)$ class.
However, in 2019 \Korandi, Tardos, Tomon, and Weidert~\cite{KorandiTTW19}
developed a new technique for analyzing $S_0,S_1$ and similar matrices.
They defined a pattern $S$ to be \emph{class-$s$ degenerate}
if it can be written $S=\left(\begin{array}{c}S'\\S''\end{array}\right)$, where at most one column has a non-zero intersection with both $S'$ and $S''$,
and $S',S''$ are at most class-$(s-1)$ degenerate; any pattern with a single row is class-0 degenerate.  Here is an example of a class-4 degenerate pattern.  It is decomposed into individual rows by sequentially making horizontal cuts, each one cutting one 
vertical line segment joining 1s in the same column.\footnote{We could also define degeneracy w.r.t. vertical cuts, i.e., $S$ is class-$s$ degenerate
if $S=(S'\, S'')$, where $S',S''$ have at most one non-zero row in common and are at most class-$(s-1)$ degenerate.  However, the \Korandi{} et al.~\cite{KorandiTTW19} method does not permit decomposing a pattern with \emph{both} vertical and horizontal cuts.}
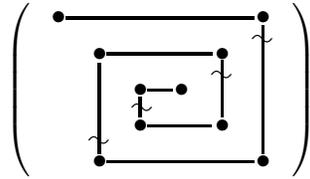
\[
\left(\begin{array}{llllll}\;
\bu\zero{\rb{0.7}{\rule{2.5cm}{0.3mm}}}&&&&&\bu\\
&\bu\zero{\rb{0.7}{\rule{1.4cm}{0.3mm}}}&&&\bu&\zero{\hcm[-.06]\rb{2}{\AC}}\\
&&\bu\zero{\rb{0.7}{\rule{.33cm}{0.3mm}}}&\bu&\zero{\hcm[-.06]\rb{2}{\AC}}&\\
&\zero{\hcm[-.06]\rb{-2}{\AC}}&\bu\zero{\hcm[-.22]\rb{2.4}{\AC}}\zero{\hcm[-.11]\rb{2}{\rotatebox{90}{\rule{.27cm}{0.3mm}}}}\zero{\rb{0.7}{\rule{.85cm}{0.3mm}}}&&\bu\zero{\hcm[-.11]\rb{2}{\rotatebox{90}{\rule{0.75cm}{0.3mm}}}}&\\
&\bu\zero{\hcm[-.11]\rb{2}{\rotatebox{90}{\rule{1.2cm}{0.3mm}}}}\zero{\rb{0.7}{\rule{1.95cm}{0.3mm}}}&&&&\bu\zero{\hcm[-.11]\rb{2}{\rotatebox{90}{\rule{1.7cm}{0.3mm}}}}\;
\end{array}\right)
\]

They proved that every class-$s$ degenerate $S$ has
\[
\Ex(S,n) \leq n\cdot 2^{O(\log^{1-\frac{1}{s+1}} n)}=n^{1+o(1)}.
\]
As a consequence, $\Ex(S_0,n),\Ex(S_1,n) \leq n2^{O(\log^{2/3} n)}$,
and by being more careful with the analysis of $S_0$, they proved $\Ex(S_0,n)\leq n2^{O(\sqrt{\log n})}$.  The results of \Korandi{} et al.~\cite{KorandiTTW19} did not directly challenge the Pach-Tardos conjecture (\cref{conj:PachTardos}), and the authors characterized
their results as taking a step towards 
\emph{affirming} \cref{conj:PachTardos}.

Pettie and Tardos~\cite{PettieT24} 
introduced a class of matrices $(A_t)$ such that 
$A_t$ is $B_t$-free and
$\|A_t\|_1 = \Theta(n(\log n/\log\log n)^t)$.
The \emph{box} pattern $B_t$ is a $2t \times (2t+1)$ matrix, 
where the first and last rows form a reflection of $P_{t+1}$
and the second and last columns form a 
rotation of $P_t$.

\[
B_t = \left(\begin{array}{ccccccccc}
&\bu &&\bu&\cdots&\bu&&\bu&\bu\\
&&&&&&&&\bu\\
&\bu&&&&&&&\\
&&&&&&&&\bu\\
&\bu&&&&&&&\vdots\\
&\vdots&&&&&&&\bu\\
&\bu&&&&&&&\\
\bu&&\bu&\cdots&\bu&&\bu&&\bu
\end{array}\right).
\]
Hence
\[
\Ex(B_t,n) = \left\{\begin{array}{l}
\Omega(n(\log n/\log\log n)^t) \\
O(n\log^{4t-3} n),
\end{array}\right.
\]
where the upper bound follows from iterated application of \cref{lem:PachTardos-reductions}(B).  The Pettie-Tardos~\cite{PettieT24} 
lower bounds are the highest obtainable lower bounds that are \emph{consistent} with the Pach-Tardos conjecture.

\subsection{Extensions and Variants the Pach-Tardos Conjecture}

\Furedi, Jiang, Kostochka, Mubayi, and Verstra\"ete~\cite{FurediJKMV21} studied forbidden
patterns in ordered $r$-uniform hypergraphs.  They
made a conjecture extending \cref{conj:PachTardos}.

\begin{conjecture}[{\cite[Conjecture B]{FurediJKMV21}}]\label{conj:Furedietal-ordered-hypergraph}
Let $F$ be any $r$-uniform forest with interval chromatic number $r$.  Then the maximum number of edges in a vertex-ordered $r$-uniform hypergraph with no subgraph order-isomorphic to $F$
is $O(n^{r-1}\log^{c(F)} n)$, 
for some constant $c(F)$.
\end{conjecture}

One can think of \cref{conj:Furedietal-ordered-hypergraph} as a collection of separate conjectures for each $r\geq 2$. 
For $r=2$ we get back an equivalent form of~\cref{conj:PachTardos}; see~\cite[Theorem 2]{PachTardos06}.


Shapira and Yuster~\cite{ShapiraY17} considered an extremal problem on augmented tournaments.  A \emph{tournament} is a complete graph with ${n\choose 2}$ edges, each of which is assigned some direction.  
A $t$-augmented tournament has $t$ extra directed edges, i.e., $t$ pairs of vertices $\{u,v\}$ have both edges $(u,v),(v,u)$.  
Shapira and Yuster defined $t(n,H)$ to be the minimum number such that
\emph{any} $n$-vertex, $t(n,H)$-augmented tournament contains a 
subgraph isomorphic to the tournament $H$.  
They defined a notion of ``tournament forest'' and made an analogue of the Pach-Tardos conjecture.

\begin{conjecture}[{\cite[Conjecture 1]{ShapiraY17}}]\label{conj:Shapira-Yuster}
For any tournament forest $H$ there exists a constant $c_H$
such that $t(n,H) = O(n\log^{c_H} n)$.
\end{conjecture}

Moreover, they proved that \cref{conj:Shapira-Yuster} is equivalent to the Pach-Tardos conjecture.

\begin{theorem}[{Shapira and Yuster~\cite[Theorem 1]{ShapiraY17}}]\label{thm:Shapira-Yuster}
\cref{conj:PachTardos,conj:Shapira-Yuster} are equivalent.
\end{theorem}

\subsection{New Results}\label{sec:newresults}

Our main result is a \emph{refutation} of the Pach-Tardos conjecture (\cref{conj:PachTardos}) in both its \emph{weak} and \emph{strong} forms. As stated above, the Pach-Tardos conjecture is equivalent to both the $r=2$ case of the F\"uredi et al.\ conjecture
(\cref{conj:Furedietal-ordered-hypergraph}) and the Shapira-Yuster conjecture (\cref{conj:Shapira-Yuster}) 
so both of the latter conjectures are also refuted. 
It is straightforward to modify the counterexample
of the $r=2$ case of \cref{conj:Furedietal-ordered-hypergraph} (a graph being a 2-uniform hypergraph)
to obtain counterexamples for any $r>2$;
see \cref{appendix:multidimensional}.

Specifically, we prove that the two weight-6 patterns $S_0,S_1$ 
not subject to the Pach-Tardos reductions (\cref{lem:PachTardos-reductions}) 
do not have $n\poly(\log n)$ extremal functions. Here and throughout the paper $\log$ stands for the binary logarithm.

\begin{restatable}{theorem}{mainthmPT}\label{thm:PT-refutation}
    $\Ex(S_0,n), \Ex(S_1,n) \geq  n2^{\sqrt{\log n}-O(\log\log n)}$.
\end{restatable}

\cref{thm:PT-refutation} matches \Korandi{} et al.'s~\cite{KorandiTTW19} upper bound 
$\Ex(S_0,n) \leq n2^{O(\sqrt{\log n})}$, 
up to the hidden constant in the exponent, 
which happens to be $4$ in the upper bound 
rather than the $1$ in the lower bound.
We extend Theorem~\ref{thm:PT-refutation} in two directions. 
First, we show that the matrices constructed for the proof of this theorem avoid a large class of matrices beyond $S_0$ and $S_1$; 
see Theorem~\ref{thm:A-avoids-more}. Second, we modify the construction to increase its weight to $n2^{C\sqrt{\log n}}$ for any desired constant $C\geq 1$ and show that the matrices
obtained still avoid some acyclic patterns;
see Theorem~\ref{thm:denser-construction}.

\medskip 

We must admit that we \emph{did not} 
specifically set out to disprove the Pach-Tardos conjecture.  
Our initial aim was simply to better understand variations on 
the construction of~\cite{PettieT24}, and to understand 
simple, structured patterns like the $(P_t)$ class.
This effort was also very successful.

\begin{theorem}\label{thm:Pt-upper-lower-bound}
    For every $t\ge 2$, $\Ex(P_t,n) = \Theta(n(\log n/\log\log n)^t)$.
\end{theorem}

Both the upper and lower bounds of \cref{thm:Pt-upper-lower-bound} are new.
No lower bound better than 
$\Ex(P_t,n) \geq \Ex(P_1,n) = \Omega(n\log n)$~\cite{FurediH92,BienstockG91,Furedi90,Tardos05} 
was previously known, and the best upper bound
was $O(n\log^t n)$, which follows from iterated application 
of \cref{lem:PachTardos-reductions}(A).
\cref{thm:Pt-upper-lower-bound} is notable in many ways.
It is the \emph{first} proof of an asymptotically sharp bound 
for any acyclic pattern with extremal function $\omega(n\log n)$;
it demonstrates that the $(\log n/\log\log n)^t$ 
density first seen in~\cite{PettieT24} is not contrived but 
a \emph{natural} phenomenon, 
and it highlights an unexpected discontinuity between $P_1$ and $P_2,P_3,\ldots$

\medskip 

Although $2^{\sqrt{\log n}}$ and $(\log n/\log\log n)^t$ look like they arise from quite different constructions, 
\cref{thm:PT-refutation,thm:Pt-upper-lower-bound}
use \emph{essentially} the same 0--1 matrix construction for their lower bounds, but under different parameterizations.

\subsection{Related and Unrelated Results}

\paragraph{Unrelated Results.} 
The function $2^{(\log n)^\delta}$ is a most fashionable 
function these days.
Kelley and Meka~\cite{KelleyM23} recently proved that Behrend's~\cite{Behrend46} 1946 construction of 3-progression-free subsets of $[N]=\{1,2,\dots,N\}$ with size $N/2^{\Theta(\sqrt{\log N})}$ is roughly the best possible.
Specifically, no 3-progression-free subset of $[N]$ has density $2^{-O(\log N)^{1/12}}$. 
This was later improved to $2^{-O(\log N)^{1/9}}$ by 
Bloom and Sisask~\cite{BloomS23}.
Abboud, Fischer, Kelley, Lovett, and Meka~\cite{AbboudFKLM24} discovered a combinatorial
boolean matrix multiplication algorithm running in 
$n^3/2^{\Omega(\log n)^{1/7}}$ time, 
improving a long line of $n^3/\poly(\log n)$-time algorithms.

\paragraph{Fine-grained Classification of Acyclic Patterns.}
Every existing analysis of an acyclic pattern $P$ has placed its extremal function in the following five-echelon hierarchy.  The first four echelons are ``natural'' inasmuch as there are lower bounds (see \cite{HS86,FurediH92} and \cref{thm:PT-refutation})
proving that certain patterns in {\bfseries Quasilinear}, {\bfseries Polylog}, and {\bfseries Near-linear} cannot be moved to a lower echelon.

\begin{description}
    \item[Linear.] $\Ex(P,n)=O(n)$.
    \item[Quasilinear.] $\Ex(P,n) = O(n2^{(\alpha(n))^{C_P}})$, 
    where $\alpha(n)$ is the inverse-Ackermann function.\footnote{There are more slowly growing functions in this class, e.g., $n\alpha(n)$~\cite{HS86,FurediH92} or $n\alpha^2(n)$~\cite{Pettie-GenDS11,Pettie15-SIDMA}.}
    \item[Polylog.] $\Ex(P,n) = O(n\log^{C_P} n)$, for some $C_P >0$.
    \item[Near-linear.] $\Ex(P,n) = n2^{O(\log^{1-\delta} n)}$, for some $\delta=\delta_P \in (0,1)$.
    \item[Polynomial.] $\Ex(P,n) = O(n^{1+C_P+o(1)})$, for some $C_P \in (0,1)$.
\end{description}

A lot of effort has been spent to understand the 
membership and boundaries of these classes.
We know of some infinite classes of {\bfseries Linear} matrices,
such as permutations~\cite{MarcusT04}, double-permutations~\cite{Geneson09}, 
and monotone patterns~\cite{Keszegh09,Pettie-SoCG11},
and even have good bounds on the leading constant factors~\cite{Fox13,CibulkaK17,GenesonPhD15} for (double) permutations.
Keszegh~\cite{Keszegh09} (see also~\cite{Geneson09,Pettie-FH11})
proved that the {\bfseries Linear} class cannot be characterized
by a finite set of minimally non-linear patterns.  
In particular,
\cref{lem:PachTardos-reductions}(A) and \cite{Geneson09} imply
that every pattern in the infinite sequence $(G_t)$ has 
$\Ex(G_t,n)=\Theta(n\log n)$ and there is an infinite sequence 
$(H_t), H_t\prec G_t$, of minimally non-linear patterns w.r.t.~$\prec$~\cite{Keszegh09,Geneson09}.\footnote{I.e., it is not known if $G_t$ is itself minimally non-linear.}

\begin{align*}
G_0 &= {\scriptsize \left(\begin{array}{cccc}
		& \bu	&\bu	& 		\\
		&		& 	 	& \bu		\\
		&		& 	 	& \bu		\\
\bu\end{array}\right)},
\hcm[.25]
G_1 = {\scriptsize \left(\begin{array}{ccccccc}
	&\bu	&\bu\\
	&	&	&	& &\bu\\
	&	&	&	&\bu &\\
\bu\\
	&	&	&	&	&	&\bu\\	
	&	&	&	&	&	&\bu\\
	&	&	&\bu
\end{array}\right)},
\hcm[.25]
G_2 = {\scriptsize \left(\begin{array}{cccccccccc}
	&\bu	&\bu\\
	&	&	&	& &\bu\\
	&	&	&	&\bu&\\
\bu\\
	&	&	&	&	&	&	&&\bu\\
	&	&	&	&	&	&	&\bu&\\
	&	&	&\bu\\
	&	&	&	&	&	&	&	&	&\bu\\	
	&	&	&	&	&	&	&	&	&\bu\\
	&	&	&	&	&	&\bu
\end{array}\right)}.
\end{align*}

While the characterization of linear patterns seems to be elusive, F\"uredi, Kostochka, Mubayi and Verestrate \cite{FurediKMV20} gave a simple characterization of linear \emph{connected} patterns\footnote{\cite{FurediKMV20} is about connected vertex-ordered graphs but their results translate to the adjacency matrices of connected graphs too.}

A pattern $P$ is called \emph{light} if it contains exactly one 1 per column.  Light patterns are closely related to (generalized) 
Davenport-Schinzel sequences~\cite{AgarwalSharir95,Klazar02}; 
they are \emph{all} known to be in {\bfseries Quasilinear}~\cite{Klazar92,Klazar02,Keszegh09}.  
More specifically, 
for any light $P$ with two rows,
$\Ex(P,n)$ is one of
$\Theta(n)$, $\Theta(n\alpha(n))$,
$\Theta(n2^{\alpha(n)})$,
$\Theta(n\alpha(n)2^{\alpha(n)})$, or 
$n2^{(1+o(1))\alpha^t(n)/t!}$, $t\geq 2$, 
and if $P$ has more rows, $\Ex(P,n)$ is upper bounded by one of these functions or $n(\alpha(n))^{(1+o(1))\alpha^t(n)/t!}$~\cite{Pettie-DS-JACM,Pettie15-SIDMA,Pettie-GenDS11,HS86,FurediH92,Nivasch10}.
It is an open problem whether \emph{light linear} 
patterns are themselves characterized by a 
finite set of forbidden patterns.  
The only known minimally non-linear ones (w.r.t.~$\prec$ and rotation/reflection) are $Q_3,Q_3'$, with $\Ex(\{Q_3,Q_3'\},n)=2n\alpha(n)+O(n)$~\cite{Nivasch10,Pettie-DS-JACM,Pettie15-SIDMA}.
\begin{align*}
Q_3  &= {\left(\begin{array}{cccc}
\bu&&\bu\\
&\bu&&\bu
\end{array}\right),}
&
Q_3' &= {\left(\begin{array}{cccc}
\bu&&\\
&&\bu\\
&\bu&&\bu
\end{array}\right).}
\end{align*}

Aside from \cref{lem:PachTardos-reductions}'s
weight-1 column-reduction rules,
there are a couple more methods to bound the extremal function of composite patterns.
\begin{align*}
    A\oplus B &= 
\left(\,\begin{array}{ccccccc}
\cline{1-4}
\multicolumn{1}{|c}{ } &&& \multicolumn{1}{c|}{} \\
\multicolumn{1}{|c}{ } &A&& \multicolumn{1}{c|}{} &&& \\\cline{4-7}
\multicolumn{1}{|c}{ } &&\zero{\hcm[.26]{\bu}}& \multicolumn{1}{|c|}{} &&& \multicolumn{1}{c|}{}\\\cline{1-4}
&&& \multicolumn{1}{|c}{} &&B& \multicolumn{1}{c|}{}\\
&&& \multicolumn{1}{|c}{} &&& \multicolumn{1}{c|}{}\\\cline{4-7}
\end{array}\,\right),
&
A \otimes B &= \left(\,\begin{array}{ccccccc}
\cline{3-5}
&&\multicolumn{1}{|c}{} &\zero{\rb{-3}{$B$}}& \multicolumn{1}{c|}{}\\
&&\multicolumn{1}{|c}{} &&\multicolumn{1}{c|}{}\\\cline{1-3}\cline{6-7}
\multicolumn{1}{|c}{}&&\multicolumn{1}{|c|}{\bu} &&\multicolumn{1}{c|}{\bu} & \bu&\multicolumn{1}{c|}{}\\\cline{3-5}
\multicolumn{1}{|c}{}&&\multicolumn{1}{c|}{} &&\multicolumn{1}{c|}{} & &\multicolumn{1}{c|}{}\\
\multicolumn{1}{|c}{}&\zero{\rb{3}{$A$}}&\multicolumn{1}{c|}{} &\hcm[.7]&\multicolumn{1}{c|}{} &\hcm[.4]&\multicolumn{1}{c|}{\hcm[.4]}\\\cline{1-3}\cline{6-7}
\end{array}\,\right).
\end{align*}

Keszegh~\cite{Keszegh09} proved that 
$\Ex(A\oplus B,n) \leq \Ex(A,n)+\Ex(B,n)$\footnote{$A\oplus B$ applies whenever $A$ and $B$ have 1s in their SE and NW corners, respectively.}
and Pettie~\cite{Pettie-SoCG11}, generalizing~\cite{Keszegh09}, 
proved that if $\Ex(A,n)$ and $\Ex(B,n)$ 
are linear (respectively, near-linear) and $B$ is \emph{legal}\footnote{$A\otimes B$ applies when $A$ contains two consecutive 1s in its top row, $B$ has 1s in its SW and SE corners. A $B$ is \emph{legal} if it is either \emph{non-ascending} or \emph{non-descending}; see~\cite{Pettie-SoCG11}.}
then $\Ex(A\otimes B,n)$ is linear (respectively, near-linear) as well.\footnote{Specifically, if $\Ex(A,n,m)\leq nf(n,m)$
and $\Ex(B,n,m)\leq ng(n,m)$ then 
$\Ex(A\otimes B,n,m)=O(nf(n,m)g(n,m))$.}
These two composition rules put many more patterns in the 
{\bfseries Linear} and {\bfseries Polylog} classes~\cite{Keszegh09,Pettie-SoCG11}, and allow one to put 
``artificial'' matrices in the class {\bfseries Quasilinear} 
that are not light; see~\cite[Footnotes 3,4]{PettieT24}.

Every $P\in\{0,1\}^{k\times l}$ is dominated by the all-1 
$k\times l$ matrix $K_{k,l}$.  By the \Kovari-\Sos-\Turan{} theorem,
$\Ex(P,n)\leq \Ex(K_{k,l},n) = O(n^{2-1/\min\{k,l\}})$.
Very recently Janzer, Janzer, Magner, and Methuku~\cite{JanzerJMM24}
proved that if $P$ has at most $t$ 1s per row (or per column),
then $\Ex(P,n) = O(n^{2-1/t+o(1)})$, which confirmed a conjecture of 
Methuku and Tomon~\cite{MethukuT22}.  This upper bound applies to all patterns, but still provides the best known analysis for acyclic 
patterns not subject to~\Korandi~et al.'s \cite{KorandiTTW19} method.  
For example, the \emph{pretzel} and \emph{spiral} patterns below have $\Ex(T_0,n),\Ex(T_1,n)=O(n^{3/2+o(1)})$.

\begin{align*}
T_0 &= \left(\begin{array}{cccc}
\bu\zero{\hcm[0]\rb{0.7}{\rule{1.66cm}{0.25mm}}} &   &      & \bu\\
    & \bu&      &\\
\zero{\hcm[.165]\rb{2.03}{\rotatebox{90}{\rule{.72cm}{0.25mm}}}}\bu\zero{\hcm[-.07]\rb{.7}{\rule{1.12cm}{0.25mm}}} &   & \bu   &\\
    & \zero{\hcm[.165]\rb{2.05}{\rotatebox{90}{\rule{.74cm}{0.25mm}}}}\bu\zero{\hcm[-0.06]\rb{.7}{\rule{.94cm}{0.25mm}}} &     & \zero{\hcm[.073]\rb{2.05}{\rotatebox{90}{\rule{1.2cm}{0.25mm}}}}\bu\end{array}\right),
&
T_1 &= \left(\begin{array}{ccccc}
\bu\zero{\hcm[0]\rb{0.75}{\rule{0.95cm}{0.25mm}}} &&\bu\zero{\hcm[0]\rb{0.75}{\rule{0.87cm}{0.25mm}}}&&\bu\\
&\bu\zero{\hcm[0]\rb{0.75}{\rule{0.87cm}{0.25mm}}}&&\bu\\
&&&&\bu\zero{\hcm[-.115]\rb{1.99}{\rotatebox{90}{\rule{.74cm}{0.25mm}}}}\\
\zero{\hcm[.165]\rb{1.99}{\rotatebox{90}{\rule{1.2cm}{0.25mm}}}}\bu\zero{\hcm[-.07]\rb{0.75}{\rule{1.48cm}{0.25mm}}}&&&\zero{\hcm[.077]\rb{1.99}{\rotatebox{90}{\rule{.74cm}{0.25mm}}}}\bu
\end{array}\right).
\end{align*}

Little is known about the \emph{gaps between} 
{\bfseries Linear}, {\bfseries Quasilinear}, and {\bfseries Polylog}.  
For example, is there a pattern $P$ with
$\Ex(P,n)= \omega(n)$ but $o(n\alpha(n))$?
or a $P$ not in {\bfseries Quasilinear} with $\Ex(P,n) = o(n\log n)$?
Tardos~\cite{Tardos05} proved that if we consider \emph{pairs} of excludes matrices,
$\Ex(\{P_1,P_1^{|}\},n)=\Theta(n\log n/\log\log n)$
while
$\Ex(\{P_1,P_1^{\odot}\},n)=\Theta(n\log\log n)$,
where $P_1^{|}$ is the reflection of $P_1$ 
across the $y$-axis
and $P_1^{\odot}$ it its 180-degree rotation.
Whether these and other extremal functions between quasilinear and $o(n\log n)$ can be realized by a \emph{single} pattern is an open question.

\paragraph{Edge-ordered Graphs.} As we noted earlier,
0--1 patterns can be regarded as bipartite graphs 
where the \emph{vertex sets} 
on each side of the partition are ordered.
Pach and Tardos~\cite{PachTardos06} 
considered a vertex-ordered variant of \Turan's problem
where the vertex sets of the 
graph and pattern have a linear order.  
This theory
is very similar to 0--1 matrix theory, with \emph{interval chromatic number} 2 taking the role of \emph{bipartiteness}.
Gerbner et al.~\cite{GerbnerMNPTV23} introduced the 
analogous problem on edge-ordered
graphs and forbidden patterns, where \emph{order chromatic number} 2 takes the role of bipartiteness/interval chromatic number 2.
The problem of characterizing linear edge-ordered patterns seems to be equally difficult in this setting; see~\cite{KucheriyaT23b} characterizing \emph{connected} edge-ordered graphs with a linear extremal function.
Gerbner et al.~\cite{GerbnerMNPTV23} conjectured that the extremal
function of edge-ordered acyclic patterns with order chromatic number 2 is $n^{1+o(1)}$.  This conjecture was recently proved by 
Kucheriya and Tardos~\cite{KucheriyaT23}, 
with a universal upper bound of $n2^{O(\sqrt{\log n})}$.
They further conjectured that every 
such extremal function is actually 
bounded by
$O(n\poly(\log n))$ \emph{\`{a} la} Pach-Tardos.
There is no known formal connection between the
Gerbner et al./Kucheriya-Tardos conjectures 
and the \Furedi-Hajnal/Pach-Tardos conjectures 
(\cref{conj:FH-acyclic,conj:PachTardos}).
For example, the refutations of the \Furedi-Hajnal conjecture~\cite{Pettie-FH11,PettieT24} had no analogues in the edge-ordered world, so one cannot expect an \emph{automatic} way to transform \cref{thm:PT-refutation} 
from vertex-ordered to edge-ordered graphs.
In the reverse direction, \cite{KucheriyaT23} has not 
led to a universal upper bound $\Ex(P,n) \leq n2^{O(\sqrt{\log n})}$ on acyclic 0--1 patterns $P$.  Nonetheless,
\cref{thm:PT-refutation} may cause one to doubt a universal 
$O(n\poly(\log n))$ upper bound for 
pattern graphs with order chromatic number 2.

\subsection{Organization}

\cref{sect:LowerBounds} introduces a new construction of
0--1 matrices with density (i.e., average number of 1s in a row) 
$\Theta(\log n/\log\log n)^t$ or $2^{\Theta(\sqrt{\log n})}$.
It presents the full proof of \cref{thm:PT-refutation} 
and its generalizations, as well as the lower bound half
of \cref{thm:Pt-upper-lower-bound}.
\cref{sect:UpperBounds} presents the upper bound half
of \cref{thm:Pt-upper-lower-bound}.
We conclude with some open problems and conjectures in
\cref{sect:conclusion}.

\section{Lower Bounds}\label{sect:LowerBounds}

\subsection{\emph{The Good, the Bad, and the Ugly} 0--1 Matrix Construction}

We define a class of 0--1 matrices 
$A$ w.r.t.~integer parameters $b,m$.
The rows are indexed by $[m]\times[m]^b$ and the columns by $[m]^b\times\{0,1\}^b$, both ordered lexicographically. 
Here $[m] = \{1,2,\ldots,m\}$ 
is the set of the first $m$ positive integers.
If one desires a square matrix, then $m=2^b$.
We identify rows by pairs $\Row = (s,r)\in [m]\times [m]^b$
and columns by pairs $\Col = (c,i)\in [m]^b \times\{0,1\}^b$.

The matrix $A=A[b,m]$ is defined as follows\footnote{By analogy to the plot of \emph{The Good, the Bad, and the Ugly} (1966), the Ugly (column $\Col=(c,i)$) knows the location of the graveyard ($i$) while the Good (row $\Row = (s,r)$) knows the name on the grave ($s$).}
\[
A((s,r),(c,i)) = \left\{\begin{array}{ll}
    1 &     \mbox{ if $r = c + s\cdot i$,}\\
    0 &     \mbox{ otherwise.}
\end{array}\right.
\]

\begin{lemma}\label{lem:size-A}
    $\|A[b,m]\|_1 \ge2^bm^b(m/(b+1)-1)$, i.e., 
    $\Omega(m/b)$ times the number of columns.
\end{lemma}

\begin{proof}
    Pick a column $(c,i)$ uniformly at random.  
    For some $s\in [m]$, 
    there exists a row $(s,r)$ such that $A((s,r),(c,i))=1$
    iff, for all $u\in [b]$ with $i(u)=1$, $r(u) = c(u) + s$
    does not ``overflow'' the range $[m]$.  
    We have
    $\E(\max_{u\in [b]} c(u)) \le m - m/(b+1)+1$.  Thus, a
    column of $A$ has at least $m/(b+1)-1$ 1s on average.
\end{proof}

\subsubsection{Matrices via Alternating Coordinate Offsets}

To construct $P_t$-free matrices we need a small 
modification to the $A[b,m]$ construction called $A_t = A_t[b,m]$.
The row-set of $A_t$ is the same as in $A[b,m]$, namely $[m]\times[m]^b$, but we only keep the following subset of the columns:
$\{(c,i)\in[m]^b\times\{0,1\}^b  \:\mid\: \|i\|_1 = t\}$.
Thus, to form a square $A_t[b,m]$ one would set $m={b\choose t}$.
The rows and columns of $A_t$ are also ordered lexicographically.

Define the vectors $\ieven,\iodd \in \{0,1\}^b$ 
for $i\in\{0,1\}^b$ as follows.
\begin{align*}
    \ieven(u) &= i(u)\cdot \left(\left(1+\sum_{v\geq u} i(v)\right) \MOD 2 \right),
&    \iodd(u) &= i(u) \cdot \left(\left(\sum_{v\geq u} i(v)\right) \MOD 2\right).
\end{align*}

Clearly, $i=\ieven + \iodd$, 
with $\iodd$ containing the last, 3rd last,
5th last 1s of $i$, 
and $\ieven$ containing the 2nd last, 4th last, 6th last 1s of $i$.
\begin{align*}
    A_t((s,r),(c,i)) &= \left\{\begin{array}{ll}
        1 & \mbox{ if $r = c  + s\cdot \ieven + (m+1-s)\iodd$,}\\
        0 & \mbox{ otherwise.}
    \end{array}\right.
\end{align*}  

\begin{remark}
    It is possible to show that if one restricts $A=A[b,m]$ 
    to columns $(c,i)$ with $\|i\|=t$, then $A$ is $P_{2t}$-free.  
    The reason for introducing $A_t$ and the alternating offsets 
    $s, m+1-s$ is to prove $P_t$-freeness.
\end{remark}

\begin{lemma}\label{lem:size-Aprime}
$\|A_t[b,m]\|_1 = \Omega(m^{b+1}{b\choose t}/(t2^t))$, 
i.e., 
$\Omega(m/(t2^t))$ times the number of columns.
\end{lemma}

\begin{proof}
Pick a column $(c,i)$ uniformly at random. 
For a fixed $s\in[m]$, there exists a row $(s,r)$ with $A_t((s,r),(c,i))=1$ iff
$c(u) + s \in [m]$ for all $u$ with $\ieven(u)=1$ and
$c(u) + m+1-s \in [m]$ for all $u$ with $\iodd(u)=1$.
Thus, $(c,i)$ will be incident to some row $(s,r)$ for \emph{every} 
$s\in [(1/2-\epsilon)m,(1/2+\epsilon)m]$
with probability at least $(1/2-\epsilon)^t$.  
Taking $\epsilon=1/t$,
the columns have $\Omega(m/(t2^t))$ 1s on average.
\end{proof}

We consider elements of $[m]^b$ as $b$-dimensional 
integer vectors and add/subtract them accordingly.
When applied to vectors, `$<$' denotes lexicographic order.
Note that lexicographic ordering
makes these vectors into a linearly ordered group.

\begin{lemma}[Simple Properties]\label{lem:A-simple-properties}
Let $A^* = A[b,m]$ or $A_t[b,m]$.  
Consider an occurrence in $A^*$ of the following patterns.

\[
        \begin{array}{c}
        \Row_0\\ \Row_1
        \end{array}
        \left(\begin{array}{c}
        \zero{\rb{4.5}{\hcm[-.1]$\Col_0$}}\bu\\
        \bu
        \end{array}\right),
\hcm[3.5]
        \begin{array}{c}
        \Row_0\\
        \end{array}
        \left(\begin{array}{cc}
        \zero{\rb{4.5}{\hcm[-.1]$\Col_0$}}\bu & \zero{\rb{4.5}{\hcm[-.1]$\Col_1$}}\bu\\
        \end{array}\right),
\]
where $\Row_0=(s_0,r_0)$ and $\Col_0=(c_0,i_0)$, etc.
Then on the left, $s_0 < s_1$, and on the right $c_0<c_1$.
If $u$ is the position of the first non-zero of $c_1-c_0$ then 
$i_0(u)=1$. If $A^*=A[b,m]$, then $(c_1-c_0)(u)=s_0$,
whereas if $A^*=A_t[b,m]$ then $(c_1-c_0)(u)=s_0$ if $\ieven_0(u)=1$ and $(c_1-c_0)(u)=m+1-s_0$ if $\iodd_0(u)=1$.
\end{lemma}

\begin{proof}
    Observe that if $A((s,r),(c,i))=1$ then `$r$' is uniquely determined by $s,(c,i)$ and `$i$' is uniquely determined by $(s,r),c$.
    
    \underline{Left.} Rows are ordered primarily by their $s$ component, so we must have $s_0\le s_1$. From the observation above, $s_0=s_1$ is not possible.
    
    \underline{Right.} As above, we must have $c_0\le c_1$ and $i_0=i_1$ is not possible, from the observation above.  
    If $A^* = A[b,m]$ we have 
    $c_1 - c_0 = s_0(i_0-i_1)$. Here all coordinates of $i_0-i_1$ are in $\{-1,0,1\}$. Now $i_0\ne i_1$ implies that there is a first non-zero position $u$ of $i_0-i_1$ and $c_0\le c_1$ implies that $(i_0-i_1)(u)=1$ hence $(c_1-c_0)(u)=s_0$.
    
    If $A^*=A_t[b,m]$, then we have $c_1-c_0=s_0(\ieven_0-\ieven_1)+(m+1-s_0)(\iodd_0-\iodd_1)$.  Then $(i_0-i_1)(u)=1$ as before,
    and depending on whether $\ieven_0(u)=1$ or $\iodd_0(u)=1$, we will see either $(c_1-c_0)(u) = s_0$ or $m+1-s_0$.
\end{proof}

\subsection{Refutation of the Pach-Tardos Conjecture}\label{sect:PachTardos-refutation}

We now recall and prove \cref{thm:PT-refutation}.

\mainthmPT*

\begin{proof}
    We shall prove that for all $b$ and $m$, $A=A[b,m]$ is 
    $\{S_0,S_1\}$-free.  The lower bound follows by setting 
    $m=2^b$, in which case $A$ is an $n\times n$ matrix,
    $n=2^{b^2+b}$ with $\|A\|_1 = \Omega(n(m/b))$.  Here $b=\sqrt{\log n}-O(1)$.
    
    Let us assume that $S_0$ is contained in $A$, 
    embedded in the rows and columns as indicated below.
    
\begin{align*}
    S_0 &= \begin{array}{c}
        \Row_0\\ \Row_1\\ \Row_2
        \end{array}
        \left(\begin{array}{ccccccc@{\hspace{.3cm}}}
        \zero{\rb{4}{$\Col_0$}}\bu && \zero{\rb{4}{$\Col_1$}}\phantom{\bu} && \zero{\rb{4}{$\Col_2$}}\bu && \zero{\rb{4}{$\Col_3$}}\phantom{\bu}\\
        \bu &&&&&& \bu\\
        &&\bu&&&&\bu
    \end{array}\right),
\end{align*}
where $\Row_0 = (s_0,r_0)$ and 
$\Col_0=(c_0,i_0)$, etc. 

Consider the 
differences $x=c_3-c_0$, $y=c_2-c_0$, and $z=c_3-c_1$. 
We clearly have $y \le x$, $z \le x$, and $y+z \ge x$. 
Note that among the lexicographically ordered non-zero integer vectors, this implies that either
(i) $x$ and $y$ agree on the position and value of their first non-zero coordinate,
(ii) $x$ and $z$ agree on the position and value of their first non-zero coordinate, or
(iii) the position of the first non-zero coordinate is the same for $x$, $y$ and $z$ and its value is strictly smaller for $y$ and $z$ than for $x$.
\cref{lem:A-simple-properties}, applied separately to the three rows, tells us that the first non-zero coordinates of $x$, $y$, and $z$ are $s_1$, $s_0$ and $s_2$, respectively. Applying \cref{lem:A-simple-properties} to columns $\Col_0$, and $\Col_3$, 
we obtain that $s_0 < s_1 < s_2$, making cases (i), (ii), and (iii) impossible.

\medskip

Note that $S_1$ is obtained from $S_0$ by swapping the first two rows. 
Without changing the definition of $x,y,z$, we still have
$y,z \le x$ and $y+z \ge x$.  \cref{lem:A-simple-properties}
implies the first non-zero coordinates in $x,y,z$ are
$s_0,s_1,s_2$, respectively, and that $s_0 < s_1,s_2$.  
Once again, cases (i), (ii), and (iii) are all impossible.
\end{proof}

\subsection{Extensions of \cref{thm:PT-refutation}}

In this section we extend \cref{thm:PT-refutation}'s lower bound 
on $\Ex(S_0,n)$ in two directions.
In \cref{sect:covering-pattern} we show that the argument of
\cref{thm:PT-refutation} can be applied to any 
\emph{covering pattern}, 
a class whose simplest members are $S_0$ and $S_1$.
In \cref{sect:hierarchy-of-S0-like-patterns} 
we show that there is a class of $S_0$-like patterns whose 
extremal functions can reach 
$n2^{C\sqrt{\log n}}$ for any desired constant $C$.

\subsubsection{Covering Patterns}\label{sect:covering-pattern}

\cref{def:covering-patterns} specifies the class of \emph{covering patterns}.

\begin{definition}\label{def:covering-patterns}
    Let $M\in\{0,1\}^{\alpha \times \beta}$ be an acyclic pattern
    with rows indexed with $0\le i<\alpha$ and columns $0\le j<\beta$.
    $M$ is a \emph{covering pattern} 
    if there is a distinguished row $k^*$ satisfying the following properties.
    \begin{enumerate}
        \item $M(k^*,0)=M(k^*,\beta-1)=1$, i.e.,  row $k^*$ has 1s in the first and last columns.
        \item Let $J$ be the set of row indices, excluding $k^*$, 
        with at least two 1s.  
        $J$ contains at most one element $j_0<k^*$.
        If such an element $j_0\in J$ exists, then some column contains two 1s in rows $[j_0,k^*]$. For any element $j\in J$ with $j>k^*$ some column contains two 1s in rows $[k^*,j]$.
        
        \item For $l\in J$ define $\first(l),\last(l)$ to be the column indices 
        of the first and last 1s in row $l$.  
        Then the \emph{real} intervals $[\first(l),\last(l)]$ cover the entire range, i.e., we have $\bigcup_{l\in J} [\first(l),\last(l)] = [0,\beta-1]$. 
    \end{enumerate}
\end{definition}

$S_0$ and $S_1$ satisfy \cref{def:covering-patterns}, 
as do $S_2,S_3$.  
Note that $S_0,S_2$ have a $j_0 < k^*$ whereas $S_1,S_3$ do not.
\begin{align*}
    S_2 &= \left(\begin{array}{ccccccc}
        &\bu & && &\bu  \\
        &&&\bu\\
        \bu&&&\bu&&&\bu\\
        \bu && \bu\\
        &&&&\bu&&\bu
    \end{array}\right),
&
    S_3 &=  \left(\begin{array}{cccccccc}
        \bu&&&&&&&\bu\\
        &&&\bu\\
        &&&\bu&&&\bu\\
        &&&&&\bu&&\bu\\
        &\bu&&&\bu\\
        \bu&&\bu
    \end{array}\right).
\end{align*}

\begin{theorem}\label{thm:A-avoids-more}
    For every covering pattern $M$ satisfying \cref{def:covering-patterns},
    $A[b,m]$ is $M$-free, hence $\Ex(M,n) \geq n2^{\sqrt{\log n}-O(\log\log n)}$.
\end{theorem}

\begin{proof}
Suppose, towards obtaining a contradiction, that $A=A[b,m]$ contains $M$,
with row $l$ of an instance of $M$ labeled $(s_l,r_l)$ and column $j$
labeled $(c_j,i_j)$.
\cref{def:covering-patterns}(2) implies that 
all members of $\{s_l \mid l\in J\}$ are different from $s_{k^*}$, 
and all but at most one are strictly greater than $s_{k^*}$.

Define $x=c_{\beta-1}-c_0$, and for $l\in J$, $y_l = c_{\last(l)}-c_{\first(l)}$.
Suppose the position of the first non-zero in $x$ is $u$.
By \cref{lem:A-simple-properties}, $x(u)=s_{k^*}$, and 
the first non-zero in the vector $y_l$ is $s_l$. 
We have $y_l \le x$ for all $l\in J$ hence $y_l(u)=0$ for all $l>k^*$.
(If $l<k^*$, it is possible that $y_l(u)=s_l < s_{k^*}$.)
The covering 
property \cref{def:covering-patterns}(3) 
implies that $\sum_{l\in J} y_l \ge x$, 
but this cannot be satisfied since $\sum_{l\in J} y_l(u) < x(u) = s_{k^*}$.
We conclude that $A$ is $M$-free and by \cref{lem:size-A}, 
$\Ex(M,n)\geq n2^{\sqrt{\log n}-O(\log\log n)}$.
\end{proof}

\subsubsection{A Hierarchy of Patterns $S_0^{(t)}$}\label{sect:hierarchy-of-S0-like-patterns}

We prove that for every $C\geq 1$, 
there exists an acyclic pattern 
$S_0^{(t)}$ such that 
$\Ex(S_0^{(t)},n) \geq n2^{C\sqrt{\log n}}$.
The 0--1 matrix construction uses a small generalization of Behrend's arithmetic progression-free sets.

\begin{lemma}[Cf.~Behrend~\cite{Behrend46}]\label{lem:Behrend}
    For any $h\geq 2$, there exists a subset $S\subset [N]$ with 
    $|S|=N/2^{O(\sqrt{\log h\log N})}$ 
    such that there are no non-trivial solutions to
    $\alpha s_0 + \beta s_1 + \gamma s_2 = 0$, 
    with $s_0,s_1,s_2\in S$ and integers
    $-h\le\alpha,\beta,\gamma\le h$.\footnote{A solution is trivial if it exists for any non-empty $S$, that is, $\alpha+\beta+\gamma=\alpha\beta(s_0-s_1)=\beta\gamma(s_1-s_2)=\gamma\alpha(s_2-s_0)=0$.}
\end{lemma}

\begin{proof}
    Let $V\subset \{0,\ldots,d-1\}^D$ 
    be a subset of vectors with $|V|\geq d^{D-2}/D$ having a 
    common $\ell_2$-norm.  Let us obtain $S$ from $V$ by prefixing each vector in $V$ with a `1' and then interpreting them as $(D+1)$-digit 
    integers in base $2hd$. In formula we have
    \[
    S=\left\{(2hd)^D+\sum_{i=0}^{D-1}v_i(2hd)^i \;\;\middle|\;\; (v_{D-1},\dots,v_1,v_0)\in V\right\}.
    \]
    We set $d=\floor{2^{\sqrt{\log h\log(N/2)}-\log(2h)}}$ and $D=\floor{\sqrt{\log(N/2)/\log h}}$, 
    so for all $s\in S$ we have $s<2(2hd)^D \leq N$.  
    Expressed in terms of $h,N$,
    $|V|=|S|\geq d^{D-2}/D = N/2^{O(\sqrt{\log h\log N})}$.
    
    Now suppose there is a solution to 
    $\alpha s_0 + \beta s_1 + \gamma s_2=0$ with integers
    $-h\le \alpha,\beta,\gamma\le h$. We need to show this is a trivial solution. With a slight abuse of notation, we will identify $s_0,s_1,s_2$ with the sequence of $D+1$ digits in their base $2hd$ expression. Note that these sequences are all obtained from elements in $V$ by putting an extra 1 in front of them, so they have the same $\ell_2$-norms, say $r$. By symmetry, we may assume that $\gamma \leq 0 \leq \alpha,\beta$ and consider the equality $\alpha s_0+\beta s_1=|\gamma|s_2$ that holds on the level of integers. Due to the base being much larger than any of the digits, it must also hold for the corresponding sequence of digits. The first digit of $\alpha s_0+ \beta s_1$ is $\alpha+\beta$, while the first digit of $|\gamma|s_2$ is $|\gamma|$, so we must have $\alpha+\beta=|\gamma|$, $\alpha+\beta+\gamma=0$. From $\|s_0\|_2=\|s_1\|_2=\|s_2\|_2=r$ we obtain $\|(|\gamma|s_2)\|_2 = |\gamma| r$ but if $s_0\neq s_1$ and $\alpha$, $\beta$ are positive, then $\|\alpha s_0 + \beta s_1\|_2 < (\alpha+\beta)r = |\gamma| r$, a contradiction. So if $\alpha,\beta>0$, we have $s_0=s_1$, and from $(\alpha+\beta)s_0 =|\gamma|s_2$, we also have $s_0=s_1=s_2$ and the solution is trivial. If one or both of $\alpha$ or $\beta$ is zero, we similarly get that $\alpha s_0+\beta s_1+\gamma s_2=0$ 
    constitutes a trivial solution.
\end{proof}

\begin{theorem}\label{thm:denser-construction}
    For each $t\geq 2$, 
    $\Ex(S_0^{(t)},n) \geq n2^{(1-o(1))\sqrt{\log t\log n}}$, 
    where $S_0^{(t)}$ is defined to be
    
\begin{align*}
    S_0^{(t)} &= \begin{array}{c}
        \Row_0\\ \Row_1\\ \Row_2
        \end{array}
        \left(\begin{array}{ccccccccccccc@{\hspace{.3cm}}}
        \zero{\rb{4}{$\Col_0$}}\bu 
          && \zero{\rb{4}{$\Col_1$}}\phantom{\bu} 
          && \zero{\rb{4}{$\Col_2$}}\bu 
          && 
          && \zero{\hcm[-.33]\rb{4}{$\Col_{2t-3}$}}\phantom{\bu}
          && \zero{\hcm[-.28]\rb{4}{$\Col_{2t-2}$}}\bu
          && \zero{\hcm[-.28]\rb{4}{$\Col_{2t-1}$}}\phantom{\bu}\\
        \bu &&&&&&\ldots&&&&&& \bu\\
        &&\bu&&&&&&\bu&&&&\bu
    \end{array}\right).
\end{align*}
Note that $S_0 = S_0^{(2)}$.
\end{theorem}

\begin{proof}
We modify the construction of $A=A[b,m]$ as follows.  
The rows and columns are identified with 
$S\times [m]^b$ and 
$[m]^b \times \{0,\ldots,t-1\}^b$, respectively,
where $S\subset [m/b]$ is the set from \cref{lem:Behrend} with size
$m^{1-o(1)}$ avoiding solutions to $\alpha s_0 + \beta s_1 + \gamma s_2 = 0$ with integers $-t+1\le\alpha,\beta,\gamma\le t-1$.
As usual $A((s,r),(c,i))=1$ iff $r = c+si$.
We make $A$ square by choosing $m$ such that $t^b = |S|$,
so $m=t^{(1+o(1))b}$, $n=m^b t^b = t^{(1+o(1))b^2}$, and $b=(1-o(1))\sqrt{\log n/\log t}$.
There are $\Theta(|S|)$ 1s per column, on average, so 
$\|A\|_1 = \Omega(n|S|) = \Omega(n2^{(1-o(1))\sqrt{\log t\log n}})$.

We now argue that $A$ is $S_0^{(t)}$-free. Assume, for a contradiction, that $A$ contains $S_0^{(t)}$ in rows $\Row_0=(s_0,r_0),\dots \Row_2=(s_2,r_2)$ and columns $\Col_0=(c_0,i_0),\dots,\Col_{2t-1}=(c_{2t-1},i_{2t-1})$. Define $x,y_k,z_k$ as 
\begin{align*}
    x &=c_{2t-1}-c_0,\\
    y_k &= c_{2k}-c_0, 
            & \mbox{ for $1\le k\le t-1$,}\\
    z_k &= c_{2t-1}-c_{2k-1},        
            & \mbox{ for $1\le k\le t-1$.}
\end{align*}
Observe that for all $k$, $y_k\le x$ and $z_k \le x$ but
\begin{align}
    y_k + z_k \ge x \ge y_{k} + z_{k+1},\label{eqn:x-yk-zk}
\end{align}
where the last inequality holds for $k<t-1$.

Note that the following analogue of Lemma~\ref{lem:A-simple-properties} holds here: (i) if a column of $A$ contains 1s in the distinct rows $(s,r)$ and $(s',r')$, then $s\ne s'$ and (ii) if the row $(s,r)$ of $A$ contains 1s in the distinct columns $(c,i)$ and $(c',i')$, then $c\ne c'$ and all coordinates of $c-c'$ are of the form $js$ with $-t+1\le j\le t-1$. In particular, (i) implies $s_0<s_1<s_2$.

Let $u$ be the position of the first non-zero coordinate of $x$. By (ii) above and $c_0\le c_{2t-1}$ we have $x(u)=js_1$ with $1\le j\le t-1$.
We use (ii) again to define $j_k$, $j_k'$ such that
\begin{align*}
    y_k(u) &= j_k s_0,\\
    z_k(u) &= j'_k s_2.
\end{align*}
Note that $y_k,z_k$ cannot have any non-zeros preceding coordinate $u$ and must have a non-negative value at position $u$.
Thus $0\le j_k\le t-1$ and for (\ref{eqn:x-yk-zk}) to be satisfied,
$1\le j'_k<j-1$ since $s_2 > s_1$.  
(This is already a contradiction when $t=2$.)
We can write \cref{eqn:x-yk-zk} as
\begin{align}
    j_k s_0 + j'_k s_2 \geq js_1 \geq j_k s_0 + j'_{k+1}s_2,\label{eqn:j-jk-jkprime}
\end{align}
Hence $t-1 \geq j > j'_1 \geq j'_2 \geq \cdots \geq j'_{t-1} \geq 1$.  
By the pigeonhole principle there must exist $j'_k = j'_{k+1}$, but then 
(\ref{eqn:j-jk-jkprime}) holds with equality, meaning $S$ supports a 
non-trivial solution to $\alpha s_0 + \beta s_1 + \gamma s_2 = 0$ 
with $\alpha = j_k, \beta = -j, \gamma = j'_k$, a contradiction.
\end{proof}

\subsection{Polylogarithmic Lower 
Bounds on Alternating Matrices}

In this section we prove the lower bound 
half of \cref{thm:Pt-upper-lower-bound}.

\begin{theorem}\label{thm:Pt-lowerbound}
    For $t\geq 1$, $\Ex(P_t,n) = \Omega(n(\log n/\log\log n)^t)$.
\end{theorem}

In the context of \cref{thm:Pt-upper-lower-bound,thm:Pt-lowerbound,thm:alternating-upper-bound}, $t$ is fixed, 
and the constants hidden by $O,\Omega$ depend on $t$.
This is in contrast to \cref{lem:size-Aprime}, 
where we made the dependence on $t$ explicit.

\begin{proof}
    We prove that $A_t[b,m]$ is $P_t$-free.  The lower bound follows
    by setting $m={b\choose t}$, in which case $A_t$ is an $n\times n$
    0--1 matrix, $n=m^{b+1}={b\choose t}^{b+1}$, 
    with $\Omega(m) = \Omega((\log n/\log\log n)^t)$ 1s per column, 
    on average.

    Suppose that $A_t$ were not $P_t$-free.
    Label the rows and columns of an occurrence of $P_t$ in $A_t$ as follows,
    where $\Row_0 = (s_0,r_0), \Col_0=(c_0,i_0)$, etc.
    
\begin{align*}
    P_t &= \begin{array}{l}
    \Row_0 \\ \Row_1
    \end{array}
    \left(\begin{array}{clclclccclc}
    \zero{\rb{4}{\hcm[-.1]$\Col_0$}}\bu&& \zero{\rb{4}{\hcm[-.1]$\Col_1$}}\bu &&\zero{\rb{4}{\hcm[-.1]$\Col_2$}}\phantom{\bu}&&\zero{\rb{4}{\hcm[-.1]$\Col_3$}}\bu&&\zero{\rb{4}{$\;\;\;\;\:\Col_{t+1}$}}&&\bu\;\\
    \bu&&&&\bu&&&\rb{2.4}{$\cdots$}&\bu&
    \end{array}\right).
\end{align*}

    Define $u_j\in [b]$ to be the position of the first non-zero of 
    $c_j - c_0$.  Note that $\Col_j$ has its 1 in row $\Row_{(j+1)\MOD 2}$.  By \cref{lem:A-simple-properties}, either $\ieven_0(u)=1$ and
    $(c_j - c_0)(u_j)=s_{(j+1)\MOD 2}$ or $\iodd_0(u)=1$ and $(c_j-c_0)(u_j)=m+1 - s_{(j+1)\MOD 2}$.  In either case $(c_j-c_0)(u_j)>0$,
    so from the ordering $c_0 \le c_1 \le \cdots \le c_{t+1}$ 
    we can conclude that
    \[
    u_1 \geq u_2 \geq \cdots\geq u_{t+1}.
    \]
    Since $\|i_0\|_1 =t$, by the pigeonhole principle there must exist a $j$ with $u_j=u_{j+1}$.  In fact, there must exist such a $j$ that also
    satisfies $j$ odd and $\iodd_0(u_j)=1$ or $j$ even and $\ieven_0(u_j)=1$.\footnote{Note that $u_1 = u_2$ is the last 1 of $i_0$, 
    or $u_2 = u_3$ is the 2nd last 1 of $i_0$, etc., 
    all of which satisfy the parity criterion.}
    If $j$ is odd then 
    \begin{align*}
        c_j-c_0         &= (0,\ldots,0,m+1-s_0,\ldots)\\
        c_{j+1}-c_0     &= (0,\ldots,0,m+1-s_1,\ldots),
    \intertext{which contradicts the order $c_j \leq c_{j+1}$ as $s_0<s_1$ by \cref{lem:A-simple-properties}.  If $j$ is even then}
        c_j-c_0         &= (0,\ldots,0,s_1,\ldots)\\
        c_{j+1}-c_0     &= (0,\ldots,0,s_0,\ldots),
    \end{align*}
    which also contradicts the order $c_j \leq c_{j+1}$.  Hence $A_t$ is $P_t$-free.
\end{proof}

\section{Upper Bounds}\label{sect:UpperBounds}

\cref{thm:alternating-upper-bound} covers the upper bound half of \cref{thm:Pt-upper-lower-bound}.  
Note the condition $t\geq 2$ cannot be strengthened, as
$\Ex(P_1,n)=\Theta(n\log n)$~\cite{FurediH92,BienstockG91,Tardos05}, 
not $O(n\log n/\log\log n)$.

\begin{theorem}\label{thm:alternating-upper-bound}
For $t\geq 2$, 
    $\Ex(P_t,n) = O(n(\log n/\log\log n)^t)$.
\end{theorem}

Let $A$ be an $n\times n$, $P_t$-free matrix maximizing $\|A\|_1$. 
For each $A(r,c)=1$ we identify a number of \emph{landmark} 
column indices.

\begin{definition}[Landmark Columns]\label{def:landmark-columns}
With respect to some $A(r,c)=1$, the following column indices obey the following order whenever they exist.
\[
c < F \leq a_1 < b_1 \leq a_2 < b_2 \leq \cdots \leq a_{t-1} < b_{t-1} \leq L.
\]
\begin{itemize}
    \item $A(r,F)=A(r,L)=1$ are the first and last 1s in row $r$ following column $c$. 
    \item $A(r,a_{t-1})=A(r,b_{t-1})=1$ are consecutive 1s in row $r$ 
    such that $F\leq a_{t-1} < b_{t-1} \leq L$ and $b_{t-1} - a_{t-1}$ 
    is maximum.  
    In general, $A(r,a_j)=A(r,b_j)=1$ are consecutive 1s in row $r$ such that $F\leq a_j < b_j \leq a_{j+1}$ and $b_j-a_j$ is maximum.  (Break ties in a consistent way, say taking the first pair of consecutive 1's of maximal distance.) 
    \item If $A(r,c)$ is one of the last two 1s in row $r$, then $F$ and $L$ are not distinct indices and $a_1,b_1,\ldots,a_{t-1},b_{t-1}$ do not exist.  Whenever $F=a_j$, the indices $a_1,b_1,\ldots,a_{j-1},b_{j-1}$ do not exist.
\end{itemize}  
\end{definition}

We assign three \emph{signatures} to every 1 in $A$ 
that is not one of the last two 1s in its row.

\begin{definition}[Signatures]
Let $\zeta=\zeta(n)\ge t$ 
be a parameter (to be optimized later)
and $\epsilon \bydef \frac{1}{6(\zeta+1)(t+2)}$.
Each $A(r,c)=1$ is assigned \emph{signatures} $\sig_0(r,c),\sig_1(r,c),\sig_2(r,c)$.
Any piece of a signature that depends on undefined values (e.g., if $a_1,b_1$ do not exist)
is undefined.
\begin{description}
    \item[$\sig_0(r,c)$ : ] consists of the vector
    $(\floor{\log_\zeta(F-c)},\floor{\log_\zeta(b_1-a_1)},\ldots,\floor{\log_\zeta(b_{t-1}-a_{t-1})})$.
    \item[$\sig_1(r,c)$ : ] consists of two parts, a vector
    $(\floor{\log_{1+\epsilon}(a_1-c)},\ldots,\floor{\log_{1+\epsilon}(a_{t-1}-c)})$, and the position of each $\floor{\log_{1+\epsilon}(b_j-c)}$ relative to the elements of this vector: larger, equal, or smaller.
    \item[$\sig_2(r,c)$ : ] consists of three parts, a vector
\begin{align*}
    &\left(\floor{\log_{1+\epsilon}(b_1-c)},\; \floor{\log_{1+\epsilon}(b_2-c)},\; \ldots,\; \floor{\log_{1+\epsilon}(b_{t-1}-c)}\right),
    \intertext{the position of each $\floor{\log_{1+\epsilon}(a_j-c)}$ relative to the elements of this vector (as in $\sig_1$), and a vector}
    &\left(\min\left\{\floor{\frac{a_{2}-b_1}{(b_1-c)/2}},3t\right\},\; \ldots,\;
    \min\left\{\floor{\frac{a_{t-1}-b_{t-2}}{(b_{t-2}-c)/2}},3t\right\},\;
    \min\left\{\floor{\frac{L-b_{t-1}}{(b_{t-1}-c)/2}},3t\right\}\right).
\end{align*}
\end{description}
\end{definition}

\subsection{Structure of the Proof}\label{sect:structure-of-the-proof}

We will eventually prove that after the following 
4-step \emph{marking} procedure, there will be
no unmarked 1s in any $P_t$-free matrix $A$.

\begin{description}
\item[Step 1.] In each row, mark the last two 1s and the last 1 of each $\sig_0$-type.
\item[Step 2.] In each column, mark the last unmarked 1 of each $\sig_1$-type 
that satisfies Inequality~(\ref{eqn:F}), defined in \cref{lem:notlast-sig0-signature}.
\item[Step 3.] In each row, mark the last $t$ unmarked 1s of each $\sig_2$-type.
\item[Step 4.] In each column, mark the last $t$ unmarked 1s of each $\sig_2$-type.
\end{description}

The number of distinct $\sig_0$ signatures is $O(\log_\zeta^t n)$ while the number of $\sig_1$ and $\sig_2$ signatures is $O(\log_{1+\epsilon}^{t-1} n) = O((\zeta\log n)^{t-1})$, so the total number of marked 1s is
$O(n(\log_\zeta^t n + (\zeta\log n)^{t-1}))$.  
(Remember that $t = O(1)$ is constant.)
We choose $\zeta = (\log n)^{1/t}$ to roughly balance these 
contributions, and conclude that 
$\|A\|_1 = O(n(\log n/\log\log n)^t)$.

\subsection{The Proof}

\cref{lem:notlast-sig0-signature,lem:a-F-gap,lem:c-in-range-F-a1,lem:no-c-in-range-F-a1} will be used to prove that no unmarked 1s in $A$ 
remain after Steps 1--4.

\begin{lemma}\label{lem:notlast-sig0-signature}
Suppose that $A(r,c^*)=1$, having landmarks $(F,a_1,b_1,\ldots,L)$, 
is \emph{not} the last 1 in its row with its $\sig_0$-type.
Then for \underline{at least one} of the following $t$
inequalities, the relevant landmarks 
exist and the inequality is satisfied.
\begin{align}
a_1 - F     &> \fr{1}{3\zeta}(F - c^*), \tag{I.1}\label{eqn:F}\\
a_2 - b_1   &> \fr{1}{3\zeta}(b_1 - c^*),\tag{I.2}\label{eqn:b1}\\
            &\;\:\vdots\nonumber\\
a_{t-1} - b_{t-2}&> \fr{1}{3\zeta}(b_{t-2} - c^*),\tag{I.$t-1$}\label{eqn:bt-1}\\
L - b_{t-1}     &> \fr{1}{3\zeta}(b_{t-1}-c^*).\tag{I.$t$}\label{eqn:bt}
\end{align}
\end{lemma}

\begin{proof}
    Suppose that for $c'>c^*$, 
    $A(r,c')=1$ has the same $\sig_0$-type as $A(r,c^*)$,
    with landmarks $(F',a_1',b_1',\ldots,a_{t-1}',b_{t-1}',L')$.
    Then $c'$ lies in one of the following intervals.
    \[
    [F,a_1),[a_1,a_2),\ldots,[a_{j-1},a_{j}),\ldots, [a_{t-1},L).
    \]

    \paragraph{Case 1: $c'\in [F,a_1)$.}
    Clearly $A(r,c^*)$ and $A(r,c')$ share a suffix of the landmarks, specifically 
    $(a_1,b_1,\ldots,a_{t-1},b_{t-1},L)=(a_1',b_1',\ldots,a_{t-1}',b_{t-1}',L')$; 
    only $F' \in (F,a_1]$ is different.  
    Since $\floor{\log_\zeta(F-c^*)}=\floor{\log_\zeta(F'-c')}$, 
    it must be that $a_1-F \ge F'-c' > \fr{1}{\zeta}(F-c^*) > \fr{1}{3\zeta}(F-c^*)$.

    \paragraph{Case 2: $c' = a_j$.}  Since there are no 1s in row $r$ and the column interval $(a_j,b_j)$, $F'=b_j$ and the landmark vectors agree on the suffixes 
    $(a_{j+1},b_{j+1},\ldots,L) = (a_{j+1}',b_{j+1}',\ldots,L')$.
    We prove that at least one of Inequalities (I.1)--(I.$j+1$) is satisfied.
    If (I.1)--(I.$j$) are not satisfied, then
\begin{align}
    b_{j} - c^* &= (b_{j}-a_{j}) + (a_{j}-c^*)\nonumber\\
                &\le (b_{j}-a_{j}) + (1+\fr{1}{3\zeta})(b_{j-1}-c^*)\nonumber\\
                &\qquad \ldots\nonumber\\
                &\le (b_{j}-a_{j}) + (1+\fr{1}{3\zeta})(b_{j-1}-a_{j-1})
                    + \cdots + (1+\fr{1}{3\zeta})^{j-1}(b_1-a_1) + (1+\fr{1}{3\zeta})^j(F-c)\nonumber\\
                &< (1+\fr{1}{3\zeta})^t[(b_j-a_j) + \cdots + (b_1-a_1) + (F-c^*)].\label{eqn:bj-c}
\end{align}
    According to the common $\sig_0$ type 
    and the fact that $(a_j,b_j)=(c',F')$ we have
\begin{align}
    \floor{\log_\zeta(b_j' - a_j')} + \cdots + \floor{\log_\zeta(b_1' - a_1')} 
    &= \floor{\log_\zeta(b_j - a_j)} + \cdots + \floor{\log_\zeta(b_1-a_1)},\label{eqn:bjprime-ajprime-1}\\
    \floor{\log_\zeta(b_j'-a_j')} &= \floor{\log_\zeta(b_j-a_j)} = \floor{\log_\zeta(F'-c')} = \floor{\log_\zeta(F-c^*)}.\label{eqn:bjprime-ajprime-2}
\end{align}
    Since all the landmarks $F',a_1',\ldots,b_j'$ lie in the range $[b_j,a_{j+1}]$,
    \cref{eqn:bj-c,eqn:bjprime-ajprime-1,eqn:bjprime-ajprime-2} imply that
\begin{align*}
    2(a_{j+1}-b_j) &\geq \fr{1}{\zeta}[(b_j-a_j) + \cdots + (b_1-a_1) + (F-c^*)]\\
                &> \fr{1}{\zeta}(1+\fr{1}{3\zeta})^{-t}(b_j-c^*),
\intertext{and since $\zeta\geq t$, 
we have $(1+\fr{1}{3\zeta})^{-t} > e^{-1/3} > 2/3$.  Thus,}
    a_{j+1}-b_j &> \fr{1}{3\zeta}(b_j-c^*),
\end{align*}
i.e., Inequality (I.$j+1$) is satisfied.

\paragraph{Case 3: $c' \in [b_j,a_{j+1})$.} (Here $a_t \bydef L$.)  
The argument is identical to Case 2, 
except that since $c'$ is also in $[b_j,a_{j+1})$, 
we can substitute for
\cref{eqn:bjprime-ajprime-1,eqn:bjprime-ajprime-2}
the following inequality.
\begin{align*}
\lefteqn{\floor{\log_\zeta(b_j' - a_j')} + \cdots 
    + \floor{\log_\zeta(b_1' - a_1')} + \floor{\log_\zeta(F'-c')}}\\
    &= \floor{\log_\zeta(b_j - a_j)} + \cdots + \floor{\log_\zeta(b_1-a_1)} + \floor{\log_\zeta(F-c)}.
\end{align*}
Thus, $a_{j+1}-b_j > \fr{1}{\zeta}(1+\fr{1}{3\zeta})^{-t}(b_j-c^*) > \fr{2}{3\zeta}(b_j-c^*)$, thereby satisfying Inequality (I.$j+1$).
\end{proof}

\begin{lemma}\label{lem:a-F-gap}
Suppose $A(r_0,c^*)=A(r_1,c^*)=1$ 
have the same $\sig_1$-type for some indices $r_0<r_1$ and $c^*$,
and both satisfy Inequality~(\ref{eqn:F}).
Then $A$ contains $P_t$.
\end{lemma}

\begin{proof}
Let the landmarks for $A(r_0,c^*)$ and $A(r_1,c^*)$ be
$(F,a_1,b_1,\ldots,L)$ and $(F',a_1',b_1',\ldots,L')$, respectively.
Call \emph{$\slab(i)$} the interval of columns
$\left[c^* + (1+\epsilon)^{i}, c^* + (1+\epsilon)^{i+1}\right)$.
If the first vector in the common $\sig_1$-signature is 
$(i_1,\ldots,i_{t-1})$, then $a_j,a_j' \in \slab(i_j)$ for all $j\in [t-1]$.  (Since Inequality~(\ref{eqn:F}) is satisfied,
all $a_j,a_j'$ exist.)

\paragraph{Case 1:} Suppose that according to the second part of the common $\sig_1$-signature, there is an index $j$ such that 
$b_j,b_j' \in \slab(i_j)$.  
Then $a_j - c^*$ is at least the distance from $c^*$ to $\slab(i_j)$, which
is $\Delta \bydef (1+\epsilon)^{i_j}$, 
and $b_j-a_j$ is at most $\epsilon\Delta$,
the width of $\slab(i_j)$.
Since Inequality~(\ref{eqn:F}) is satisfied and $a_1-c^* < (1+\epsilon)^{i_1+1}$, 
$F-c^* < (1+\fr{1}{3\zeta})^{-1}(1+\epsilon)^{i_1+1}$.
Thus,
\begin{align*}
    a_j - F = (a_j-c^*) - (F-c^*) 
            &\geq \Delta - (1+\fr{1}{3\zeta})^{-1}(1+\epsilon)^{i_1+1}\\
            &\geq \Delta\left(1-\fr{(1+\epsilon)3\zeta}{3\zeta+1}\right) & \mbox{($i_1 \leq i_j$)}\\
            &> \frac{\Delta}{6(\zeta+1)}. & \mbox{($\epsilon < 1/(6\zeta)$)}
\end{align*}
By the definition of 
$a_j,b_j$, 
every interval of 
width $\epsilon\Delta$ in row $r_0$ (resp., $r_1$) 
between $F$ (resp., $F'$) and $\slab(i_j)$ contains a 1.
Thus, rows $r_0$ and $r_1$ contain an alternating pattern 
of length $(6(\zeta+1)\epsilon)^{-1} > t+1$, and together 
with column $c^*$ this forms an instance of $P_t$.
See the diagram below.

\bigskip
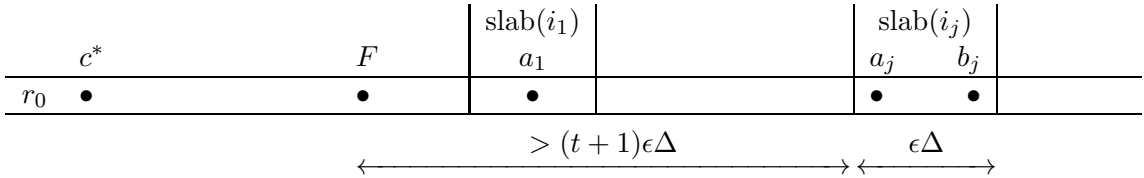
\begin{tabular}{rlcl|c|c|lr|l}
      &           &           &         & $\slab(i_1)$ & & \multicolumn{2}{c|}{$\slab(i_j)$} & \hcm[1.5]\\
      & $c^*$       &           & $F$\hcm[1]     & $a_1$ &           & $a_j$\ \  &  \ \ $b_j$ & \\\hline
$r_0$ & \bu       & \hcm[2.5] & \bu     & \bu   &   \hcm[3] & \bu\ \ & \ \ \bu& \\\hline
\multicolumn{3}{c}{\istrut{5}} 
& 
\multicolumn{6}{l}{\zero{\underleftrightarrow{\hcm[2.3] > (t+1)\epsilon\Delta\hcm[2.3]} 
\,\underleftrightarrow{\hcm[.4]\phantom{()}\epsilon\Delta\phantom{()}\hcm[.4]}}}\\
\end{tabular}

\paragraph{Case 2:} According to the second part of the common $\sig_1$-signature, 
$b_j, b'_j \not\in \slab(i_j)$ for all $j\in[t-1]$, hence $i_1 < i_2 < \cdots < i_{t-1}$.
Since $a_1 - F > (F-c^*)/(3\zeta)$ and $\epsilon < 1/(3\zeta)$, $F$ is not in $\slab(i_1)$.
Then we have an instance of $P_t$ on rows $r_0,r_1$ and columns 
$c^* < F < a_1' < a_2 < \cdots < \{a_{t-1},a_{t-1}'\} < \{b_{t-1},b_{t-1}'\}$,
where the columns selected from $\{a_{t-1},a_{t-1}'\}$ and $\{b_{t-1},b_{t-1}'\}$ 
depend on the parity of $t$.
The underlined points in the figure below form an 
instance of $P_t$ if $t$ is even.

\bigskip 
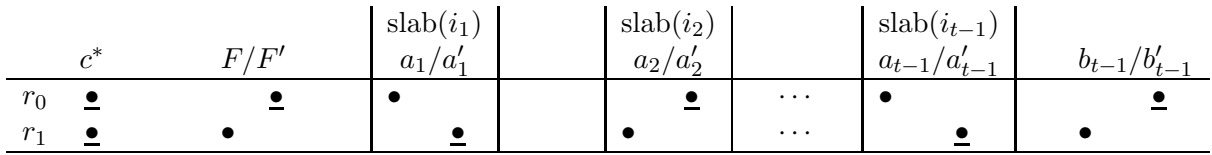
\centerline{\begin{tabular}{rclccl|cc|l|cc|l|cc|lcc}
        &    &\hcm[.7]&  & &\hcm[.4]& \multicolumn{2}{c|}{$\slab(i_1)$} & \hcm[.9] & \multicolumn{2}{c|}{$\slab(i_2)$} &  & \multicolumn{2}{c|}{$\slab(i_{t-1})$} & \hcm[.1] & &\\
        & $c^*$   &\hcm[.7]&  \multicolumn{2}{c}{$F/F'$} &\hcm[.5]& \multicolumn{2}{c|}{$a_1/a_1'$} & \hcm[1] & \multicolumn{2}{c|}{$a_2/a_2'$} &  & \multicolumn{2}{c|}{$a_{t-1}/a_{t-1}'$} & \hcm[.2] & \multicolumn{2}{c}{$b_{t-1}/b_{t-1}'$}\\\hline
$r_0$   & \ul{\bu}   &&& \ul{\bu} && \bu &      &&& \ul{\bu} &\hcm[.4]$\cdots$\hcm[.4]&   \bu &&&     & \ul{\bu}\\
$r_1$   & \ul{\bu}   && \bu &&&     & \ul{\bu}  && \bu &&\hcm[.4]$\cdots$\hcm[.4]&  & \ul{\bu} &&  \bu \\\hline
\end{tabular}}

\bigskip 

Cases 1 and 2 are exhaustive, so we conclude $A$ contains $P_t$.
\end{proof}

\begin{lemma}\label{lem:c-in-range-F-a1}
    Suppose there are two rows $r_0<r_1$ 
    and three columns $c^*<c,c'$ such that
    \begin{itemize}
        \item $A(r_0,c^*)=A(r_0,c)=A(r_1,c^*)=A(r_1,c')=1$ and all have a common $\sig_2$-type.
        \item $A(r_0,c^*)$ and $A(r_1,c^*)$ are each not the last 1 in their row with their respective $\sig_0$-types.  They \emph{do not} satisfy Inequality~(\ref{eqn:F}).
        \item Let the landmarks for $A(r_0,c^*),A(r_1,c^*)$ be $(F,a_1,b_1,\ldots,L)$ and $(F',a_1',b_1',\ldots,L')$, respectively.
        Both $c\in [F,a_1)$ and $c' \in [F',a_1')$.
    \end{itemize}
    Then $A$ contains $P_t$.    
\end{lemma}

\begin{proof}
    Since $c \in [F,a_1)$, 
    the landmarks for $A(r_0,c)$
    are $(\tilde{F},a_1,\ldots,b_{t-1},L)$, 
    i.e., they only differ from the landmarks for $A(r_0,c^*)$ in $F/\tilde{F}$.

    Let the first vector of the common $\sig_2$-signature 
    be $(i_1,\ldots,i_{t-1})$, i.e., if $i_j$ is defined then $b_j,b_j'$ exist and are in $\slab(i_j)$, where the slabs are defined as in the proof of \cref{lem:a-F-gap}.
    Since $\floor{\log_{1+\epsilon}(b_1-c^*)}=\floor{\log_{1+\epsilon}(b_1-c)}$, 
    this implies $F-c^* \leq c-c^* < \epsilon(1+\epsilon)^{i_1}$ 
    (the width of $\slab(i_1)$)
    and that
    $F$ lies strictly before $\slab(i_1)$. We similarly have $F'-c^*<\epsilon(1+\epsilon)^{i_1}$.

\paragraph{Case 1:} According to the second part of the common 
$\sig_2$-signature, 
for some index $j\in [t-1]$, $a_j,a_j'$ 
exist and are in $\slab(i_j)$.
    Then $b_j - a_j < \epsilon(1+\epsilon)^{i_j} \bydef \epsilon\Delta$ and by definition of $a_j,b_j$, every interval of width $\epsilon\Delta$ in row $r_0$ between $F$ and $\slab(i_j)$ has a 1.  The same is true in row $r_1$ (with $F'$ in place of $F$), so
    there is an alternating pattern of length 
    $(\Delta - \epsilon(1+\epsilon)^{i_1})/\epsilon\Delta \geq (1-\epsilon)/\epsilon > t+1$ between rows $r_0,r_1$, and together with column $c^*$, 
    this forms an instance of $P_t$.

    From now on we assume we are not in Case~1, so in particular, we have $i_1<i_2<\cdots<i_{t-1}$.

\paragraph{Case 2:} 
According to the common $\sig_2$-signature, 
for some index $j\in [t-1]$, $a_j$ is strictly between $\slab(i_{j-1})$ and $\slab(i_j)$.  Then $A$ contains an instance of $P_t$ on rows $r_0,r_1$ and the 
$t+2$ columns $c^*,F,b_1',\cdots,\{b_{j-1},b_{j-1}'\},\{a_j,a_j'\},\{b_j,b_j'\},\ldots,\{b_{t-1},b_{t-1}'\}$,
where the columns selected from 
$\{a_j,a_j'\},\{b_{t-1},b_{t-1}'\}$, etc.
depend on the parities of $j$ and $t$.

\paragraph{Case 3:} According to the common $\sig_2$-signature,
for every $j\in [2,t-1]$, $a_j,a_j' \in \slab(i_{j-1})$.  
By \cref{lem:notlast-sig0-signature}, at least one of Inequalities~(\ref{eqn:F})--(\ref{eqn:bt}) are satisfied, but Inequality~(\ref{eqn:F}) is not satisfied by assumption.
Since $a_j - b_{j-1}$ is less than the width of $\slab(i_{j-1})$, we cannot
satisfy any of Inequalities~(\ref{eqn:b1})--(\ref{eqn:bt-1}), hence Inequality~(\ref{eqn:bt}) is satisfied: 
$L-b_{t-1} > \fr{1}{3\zeta}(b_{t-1}-c^*)$, implying
that $L$ lies outside $\slab(i_{t-1})$ since $\epsilon < 1/(3\zeta)$. 
Thus, $A$ contains an instance of $P_t$ on rows $r_0,r_1$ 
and the $t+2$ columns 
$c^*,F,b_1',\ldots,\{b_{t-1},b_{t-1}'\},\{L,L'\}$.
\end{proof}

\begin{lemma}\label{lem:no-c-in-range-F-a1}
    Suppose there are two rows $r_0 < r_1$ 
    and three columns $c^* < c,c'$
    such that
    \begin{itemize}
        \item $A(r_0,c^*)=A(r_0,c)=A(r_1,c^*)=A(r_1,c')=1$ and all have a common $\sig_2$-type.
        \item Let the landmarks for $A(r_0,c^*),A(r_1,c^*)$ be $(F,a_1,b_1,\ldots,L)$ and $(F',a_1',b_1',\ldots,L')$, respectively (with the first few $a_i,b_i$ potentially undefined).
        For some $\ell\in [t-1]$, 
        $c \in [b_\ell,a_{\ell+1})$ 
        and $c' \in [b_\ell', a_{\ell+1}')$,
        where $a_{t}\bydef L, a_{t}' \bydef L'$.
    \end{itemize}
    Then $A$ contains $P_t$.   
\end{lemma}

\begin{proof}
    Since $c \in [b_\ell,a_{\ell+1})$, 
    the landmarks for $A(r_0,c)$ are $(F'',a_1'',\ldots,b_{\ell}'',a_{\ell+1},\ldots,a_{t-1},b_{t-1},L)$, 
    i.e., they agree on the suffix $(a_{\ell+1},\ldots,L)$ with the landmarks of $A(r_0,c^*)$.
    Let $(i_1,\ldots,i_{t-1})$ be the first component of the common $\sig_2$-signature.
    Define $\Delta_0 \in [(1+\epsilon)^{i_\ell},(1+\epsilon)^{i_\ell+1})$ 
    and $\mu>0$ as:
    \begin{align*}
        \Delta_0 &= b_\ell - c^*,\\
    \mu \Delta_0 &= a_{\ell+1} - b_\ell.
    \end{align*}

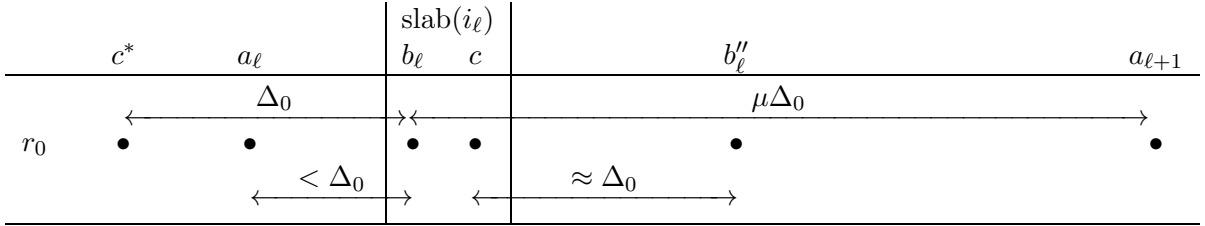
\bigskip 
\begin{tabular}{rlclcl|cc|lclc}
        &&&&          &&     \multicolumn{2}{c|}{$\slab(i_\ell)$} &&\\
        &&$c^*$&\hcm[.5]& $a_\ell$ &\hcm[1]& $b_\ell$   & $c$  &\hcm[2.2]&     $b_\ell''$       &\hcm[4.2]&   $a_{\ell+1}$\\\hline
        &\istrut{4}&\zero{\underleftrightarrow{\hcm[1.53]\phantom{\mu}\Delta_0\hcm[1.53]} 
\underleftrightarrow{\hcm[4.57]\mu\Delta_0\hcm[4.57]}} &&&&&&&&&\\
$r_0$   &&\bu&&    \bu   && \bu        & \bu   &&     \bu              &&   \bu\\
&\istrut[5]{0}&&&\zero{\underleftrightarrow{\hcm[.64]< \Delta_0\hcm[.64]} 
\hcm[.76]\underleftrightarrow{\hcm[1.33]\approx\Delta_0\hcm[1.33]}} &&&&&&&\\\hline
\end{tabular}
\bigskip 

    The relevant component of the third part of the common $\sig_2$-signature 
    is $\min\left\{\floor{\frac{a_{\ell+1}-b_\ell}{(b_\ell-c^*)/2}}, 3t\right\} 
    = \min\{\floor{2\mu},3t\}$.
    Observe that because of the first part of the common $\sig_2$-signature, 
    $b_\ell''-c \in ((1+\epsilon)^{-1}\Delta_0,(1+\epsilon)\Delta_0)$.
    Thus,
    \[
    \frac{a_{\ell+1}-b_\ell''}{b_{\ell}''-c} 
    \leq 
    \frac{\mu \Delta_0 }{(1+\epsilon)^{-1}\Delta_0}-1 = (1+\epsilon)\mu - 1.
    \]
    Note that since $\epsilon < 1/(3t)$, 
    $(1+\epsilon)\mu -1 < \mu-1/2$ whenever $\mu<3t/2$.
    Since the $\sig_2$-signatures of $A(r_0,c^*),A(r_0,c)$ 
    are identical, it must be that
    \[
    \min\{\floor{2\mu},3t\} = \min\left\{\floor{\frac{a_{\ell+1}-b_\ell}{(b_\ell-c^*)/2}}, 3t\right\} 
    =
    \min\left\{\floor{\frac{a_{\ell+1}-b_\ell''}{(b_\ell''-c)/2}}, 3t\right\}
    = 3t.
    \]
    Since $b_\ell - a_\ell < \Delta_0$ is maximum
    among distances of consecutive 1s in the range $[F,a_{\ell+1}]$, 
    there must be a 1 in row $r_0$ in every interval of width 
    $\Delta_0$ that starts in the range 
    $[a_\ell,a_{\ell+1}] \supset [c^* +(1+\epsilon)^{i_\ell}, c^* +(\mu+1)(1+\epsilon)^{i_\ell}]$.
    The same is true of row $r_1$ with respect to some 
    $\Delta_1 \in [(1+\epsilon)^{i_\ell},(1+\epsilon)^{i_\ell+1})$.
    Since $\mu\ge 3t/2 \geq t+1$, 
    there are $t+1$ alternations between 
    rows $r_0,r_1$ in the interval 
    $[c^* + (1+\epsilon)^{i_\ell}, c^* +(\mu+1)(1+\epsilon)^{i_\ell}]$.
    Together with column $c^*$ this forms an instance of $P_t$.
\end{proof}

We are now in a position to prove \cref{thm:alternating-upper-bound},
using \cref{lem:notlast-sig0-signature,lem:a-F-gap,lem:c-in-range-F-a1,lem:no-c-in-range-F-a1}.

\begin{proof}[Proof of Theorem~\ref{thm:alternating-upper-bound}]
Let $A$ be an $n\times n$ matrix.
We apply the following 4-step marking process 
and prove that either every 1 is marked,
or $A$ contains an instance of $P_t$.
\begin{description}
\item[Step 1.] In each row, mark the last two 1s and the last 1 of each $\sig_0$-type.
\item[Step 2.] In each column, mark the last unmarked 1 of each $\sig_1$-type 
that satisfies Inequality~(\ref{eqn:F}).
\item[Step 3.] In each row, mark the last $t$ unmarked 1s of each $\sig_2$-type.
\item[Step 4.] In each column, mark the last $t$ unmarked 1s of each $\sig_2$-type.
\end{description}
Suppose that there exists some $A(r_0,c^*)=1$ that is still unmarked after Steps 1--4.  
Because it is unmarked after Step 1, \cref{lem:notlast-sig0-signature} implies that 
it must satisfy at least one of Inequalities (\ref{eqn:F})--(\ref{eqn:bt}).  
If it satisfies Inequality~(\ref{eqn:F}) then Step 2 must have marked some other $A(r_1,c^*)=1$, $r_1\neq r_0$, of the same $\sig_1$-type that also satisfies Inequality~(\ref{eqn:F}).  In this case $A$ contains $P_t$, by \cref{lem:a-F-gap}.
Thus, after Step 2, we may assume that $A(r_0,c^*)$ 
does not satisfy Inequality~(\ref{eqn:F}).

Steps 3 and 4 imply that there exists
rows $r_0 < r_1 < \cdots < r_{t}$ 
and for each $i\in [0,t]$, columns
$c^* < c_{i,1} < c_{i,2} < \cdots < c_{i,t}$
such that
$A(r_i,c^*)=A(r_i,c_{i,j})=1$ all have the common $\sig_2$-type of $A(r_0,c^*)$.
Each $A(r_i,c_{i,j})$ is classified by which of the following sets 
contains $c_{i,j}$, the landmarks $(F^i,a_1^i,b_1^i,\ldots,L^i)$ being defined w.r.t.~$A(r_i,c^*)$.
\[
[F^i,a_1^i), \; \{a_1^i\}, \; [b_1^i,a_2^i), \; 
\{a_2^i\}, \; [b_2^i,a_3^i) \; \ldots, \; \{a_{t-1}^i\}, \; [b_{t-1}^i,L^i).
\]
There are only $t-1$ singleton 
sets $\{a_1^i\},\ldots,\{a_{t-1}^i\}$, 
so by the pigeonhole principle,
for each $i\in [0,t]$ there exists a $j(i) \in [t]$ such that 
$c_{i,j(i)} \in [F^i,a_1^i) \cup [b_1^i,a_2^i) \cup \cdots \cup [b_{t-2}^i,a_{t-1}^i) \cup [b_{t-1}^i,L^i)$.
By the pigeonhole principle again, 
there must exist two rows $r_{i},r_{i'}$ such that
$c_{i,j(i)},c_{i',j(i')}$ have the same classification.
If $c_{i,j(i)} \in [F^i,a_1^i)$ and 
$c_{i',j(i')} \in [F^{i'},a_1^{i'})$ then \cref{lem:c-in-range-F-a1} implies
that $A$ contains $P_t$.  
If, for some $\ell\in [t-1]$, 
$c_{i,j(i)} \in [b_\ell^i,a_{\ell+1}^i)$ and
$c_{i',j(i')} \in [b_\ell^{i'},a_{\ell+1}^{i'})$  
(by definition $a_t^i \bydef L^i$), 
then \cref{lem:no-c-in-range-F-a1} 
implies $A$ contains $P_t$.  

The cases considered above are exhaustive, 
hence if $A$ is $P_t$-free, there can be
no unmarked 1s left in $A$ after Steps 1--4.

The number of distinct $\sig_0$-, $\sig_1$-,
and $\sig_2$-signatures are 
$O(n\log_\zeta^t n), O(n\log_{1+\epsilon}^{t-1} n)$,
and $O(n\log_{1+\epsilon}^{t-1} n)$, respectively,
so the number of 1s marked by Steps 1--4 is
\begin{align*}
O(n\cdot (\log_\zeta^t n +  \log_{1+\epsilon}^{t-1} n))
&= O\left(n\cdot\left((\log n/\log\zeta)^t + (\zeta\log n)^{t-1}\right)\right). & \mbox{(Recall $\epsilon = \Theta(1/\zeta)$.)}
\end{align*}
We choose $\zeta = (\log n)^{1/t}$ and 
conclude that 
$\Ex(P_t,n) = O(n(\log n/\log\log n)^t)$.
\end{proof}

\begin{remark}
    Note that if $A$ is a rectangular $n\times m$ matrix, 
    Steps 1 and 3 mark 
    $O(n(\log_{\zeta}^t m + \log_{1+\epsilon}^{t-1} m))$ 1s
    whereas Steps 2 and 4 mark 
    $O(m\log_{1+\epsilon}^{t-1} m)$ 1s.
    Thus, if $n>m$ then we still pick $\zeta = (\log m)^{1/t}$
    and get the upper bound $\Ex(P_t,n,m) = O(n(\log m/\log\log m)^t)$,
    whereas if $n < m/\log m$, we would pick $\zeta = O(1)$
    and get the upper bound 
    $\Ex(P_t,m/\log m,m) = O(m\log^{t-1} m)$.
    This shows another qualitative 
    difference between the behavior
    of $P_1$-free and $P_t$-free, $t\ge 2$, 
    rectangular matrices.
    Pettie~\cite[Appendix A]{Pettie10a} proved that 
    for $n\geq m$, $\Ex(P_1,n,m) = \Theta(n\log_{1+n/m} m)$ and
    for $m\geq n$, $\Ex(P_1,n,m) = \Theta(m\log_{1+m/n} n)$.
    I.e., whenever $\max\{n/m, m/n\} 
    = \poly(\log(nm))$, we only lose a $\Theta(\log\log n)$-factor in the density of $P_1$-free matrices, 
    but can lose an $\Omega(\log n/\poly(\log\log n))$-factor 
    in the density of $P_t$-free matrices.
\end{remark}

\section{Conclusion}\label{sect:conclusion}

The foremost open problem in forbidden 0--1 patterns
is \emph{still} to understand the range of 
possible extremal functions for acyclic patterns. 
In the past it has been valuable to advance conjectures
of \emph{varying strength} (implausibility)~\cite{FurediH92,PachTardos06}.
Here we present \emph{weak} and \emph{strong} 
variants of the central conjecture.

\begin{conjecture}\label{conj:Acyclic}
    Let $\mathcal{P}_{\text{acyclic}}$ be the class of 
    acyclic 0--1 patterns.
    \begin{description}
    \item[Weak Form.] For all $P\in \mathcal{P}_{\text{acyclic}}$, $\Ex(P,n) = n^{1+o(1)}$.  (Cf.~\cite{Tardos05,PachTardos06}.)
    \item[Strong Form.] For all $P\in \mathcal{P}_{\text{acyclic}}$, there exists a constant $C_P$ such that $\Ex(P,n) = O(n2^{C_P \sqrt{\log n}})$.
    \end{description}
\end{conjecture}

One way to begin to tackle the {\bfseries Weak Form} of 
\cref{conj:Acyclic} is to
prove that there exists an absolute 
constant $\epsilon>0$
such that $\Ex(P,n)=O(n^{2-\epsilon})$, 
for all $P\in \mathcal{P}_{\text{acyclic}}$.
Another is to generalize 
\Korandi{} et al.'s~\cite{KorandiTTW19} method to handle 
patterns like $T_1$ (which can be decomposed using
\emph{both} vertical and horizontal cuts) and 
ultimately to patterns like the pretzel $T_0$, 
for which there is no horizontal/vertical 
cut that separates the 1s in a single column/row.
The {\bfseries Strong Form} of the 
conjecture should be considered more plausible 
in light of~\cite{KucheriyaT23}. 

Regardless of the status of \cref{conj:Acyclic}, it would be of great interest to characterize the {\bfseries Polylog} echelon: those 
$P$ with $\Ex(P,n) = O(n\log^{C_P} n)$.

\medskip 

The \emph{applications} of extremal 0--1 matrix theory in combinatorics~\cite{Klazar00,MarcusT04}, graph theory~\cite{GuillemotM14,BonnetGKTW21,BonnetKTW22}, and data structures~\cite{Pettie10a,ChalermsookGKMS15,ChalermsookGJAPY23,ChalermsookPY24} tend to use patterns that are rather low in the hierarchy, usually {\bfseries Linear} or {\bfseries Quasilinear}, 
so there is ample motivation to understand the boundary 
between these two echelons.

\begin{conjecture}\label{conj:Light-Linear}
    Let $\mathcal{Q}$ contain $Q_3,Q_3'$, and their horizontal and vertical reflections.
    If $P$ is light and $\mathcal{Q}$-free then $\Ex(P,n)=O(n)$.
    \begin{align*}
Q_3  &= {\left(\begin{array}{cccc}
\bu&&\bu\\
&\bu&&\bu
\end{array}\right),}
&
Q_3' &= {\left(\begin{array}{cccc}
\bu&&\\
&&\bu\\
&\bu&&\bu
\end{array}\right).}
\end{align*}
\end{conjecture}

\cref{conj:Light-Linear} has not even been confirmed for all 
weight-5 light patterns, though most have been classified;
see~\cite[Lemma 5.1]{Tardos05}, Fulek~\cite[Theorem 4]{Fulek09},
Pettie~\cite[Theorem 2.3]{Pettie-GenDS11}, and~\cite{MarcusT04,Pettie15-SIDMA,ChalermsookPY24}.
One way to think about classifying light $\mathcal{Q}$-free 
$P$ is to ignore any consecutive repeated columns in $P$,
then measure how many rows have at least two 1s.
If this number is zero then 
$P$ is a permutation matrix (possibly with repeated columns) and $\Ex(P,n)=O(n)$~\cite{MarcusT04,Geneson09}.
If $P$ has one such row, then the remaining 1s must be arranged in a permutation $P'$ constrained to the boxed regions in the figure below.
\[
P=\left(\begin{array}{|c|c|c|c|c|}
\cline{3-3}\multicolumn{2}{c|}{}
&
\rb{-2.5}{$P'$}
&\multicolumn{2}{c}{}\\
\multicolumn{2}{c|}{} &&\multicolumn{2}{c}{}\\\cline{3-3}\cline{1-1}\cline{5-5}
&\multicolumn{3}{c|}{} &\\
\hcm[1.2] &\multicolumn{1}{c}{\bu}&\multicolumn{1}{c}{}&\bu&\hcm[1.2]\\
&\multicolumn{3}{c|}{} &\\\cline{1-1}\cline{3-3}\cline{5-5}
\multicolumn{2}{c|}{} & \hcm[1.2] & \multicolumn{2}{c}{}\\
\multicolumn{2}{c|}{} & \hcm[1.5] & \multicolumn{2}{c}{}\\\cline{3-3}
\end{array}\right).
\]

At the other extreme, there are $\mathcal{Q}$-free patterns 
in which all but one row contain two non-consecutive 1s,
but they are all contained in the oscillating patterns $(O_t)$.
Only $O_2$ is known to be linear~\cite{FurediH92}.
Fulek~\cite{Fulek09} proved that the version of $O_3$
without the repeated column is linear.
\[
O_2 = \left(\begin{array}{cccc}
&\bu&\bu&\\
\bu&&&\bu
\end{array}\right),
\;\;\;\;
O_3 = \left(\begin{array}{cccccc}
&\bu&&&\bu&\\
\bu&&&&&\bu\\
&&\bu&\bu&&
\end{array}\right),
\;\;\;\;
O_4 = \left(\begin{array}{cccccccc}
&&&\bu&\bu&&&\\
&\bu&&&&&\bu&\\
\bu&&&&&&&\bu\\
&&\bu&&&\bu&&
\end{array}\right).
\]
There is a class of light $\mathcal{Q}$-free $P$
containing a row with \emph{three} non-consecutive 1s, 
but patterns in this class 
are highly constrained in their structure.
Each must be of the following form, up to reflection.
\[
P = \left(\begin{array}{ccccccc}
\cline{3-3}
\multicolumn{2}{c|}{}
&
\multicolumn{1}{c|}{\hcm[2]}\\
\multicolumn{2}{c|}{}
&
\multicolumn{1}{c|}{B}\\
\multicolumn{2}{c|}{}
&
\multicolumn{1}{c|}{}\\\cline{3-3}\cline{7-7}\cline{1-1}
\multicolumn{1}{|c|}{\hcm[2]}
&
\multicolumn{5}{c|}{}
&
\multicolumn{1}{c|}{\hcm[2]}\\
\multicolumn{1}{|c|}{}
&
\multicolumn{5}{c|}{}
&
\multicolumn{1}{c|}{}\\
\multicolumn{1}{|c|}{A} &\bu&&\bu&&\bu&\multicolumn{1}{|c|}{D}\\
\multicolumn{1}{|c|}{\hcm[2]}
&
\multicolumn{5}{c|}{}
&
\multicolumn{1}{c|}{}\\
\multicolumn{1}{|c|}{\hcm[2]}
&
\multicolumn{5}{c|}{}
&
\multicolumn{1}{c|}{}\\\cline{1-1}\cline{5-5}\cline{7-7}
\multicolumn{4}{c|}{} & \multicolumn{1}{c|}{\hcm[2]}\\
\multicolumn{4}{c|}{} & \multicolumn{1}{c|}{C}\\
\multicolumn{4}{c|}{} & \multicolumn{1}{c|}{}\\\cline{5-5}
\end{array}\right),
\]
where $B,C\neq 0$, $A,D$ are permutations avoiding the weight-3 row, and $B,C$ are $\left(\begin{array}{@{}c@{}c@{}c@{}}\bu&&\bu\\&\bu\end{array}\right)$-free
and 
$\left(\begin{array}{@{}c@{}c@{}c@{}}&\bu\\\bu&&\bu\end{array}\right)$-free,
respectively.  In particular, there cannot be a second row
(intersecting $A,B,C,$ or $D$) that has three non-consecutive 1s.
There are no light $\mathcal{Q}$-free 
patterns containing four non-consecutive 1s.

\newcommand{\etalchar}[1]{$^{#1}$}


\appendix

\section{Refutation of the F\"uredi et al.\ Conjecture for Hypergraphs}\label{appendix:multidimensional}

We start with defining some simple terms. For the standerd definition of \emph{vertex-ordered hypergraph} and \emph{order-isomorphism}, see \cite{FurediJKMV21}. Define $\ExH^r(H,n)$ to be the maximum number of edges 
in an $r$-uniform, vertex-ordered 
hypergraph not containing any 
subgraphs order-isomorphic 
to an $r$-uniform, vertex-ordered $H$.
The \emph{interval chromatic number} of $H$, denoted $\chi_<(H)$, is the smallest number of intervals
that $V(H)$ can be partitioned into (w.r.t.~the order on $V(H)$)
such that each hyperedge intersects $r$ distinct intervals.
As in the \Erdos-Stone-Simonovits theorem~\cite{ErdosS46,ErdosS66},
$\ExH^r(H,n) = \Theta(n^r)$ if $\chi_<(H)>r$.
A hypergraph $H$ is called a \emph{forest} if
there is a peeling order $e_1,\ldots,e_{|E(H)|}$
where for all $i<|E(H)|$, there exists an $h>i$ such that
$e_i \cap \bigcup_{j>i} e_j \subset e_h$.

A closely related concept is the generalization of extremal function $\Ex$ from 2-dimensional matrices to $r$-dimensional matrices.  If $P\in\{0,1\}^{k_1\times \cdots \times k_r}$, 
$A\in\{0,1\}^{n_1\times \cdots \times n_r}$, we say $P\prec A$ if there are index sets $(I_i)$, $I_i\subseteq [n_i]$, $|I_i|=k_i$, 
such that $P$ is entry-wise dominated by the submatrix $A[I_1,\ldots,I_r]$ of $A$ restricted to $I_1,\ldots,I_r$.  
Define 
$\Ex^r(P,n) = \max\{\|A\|_1 \:\mid\: A\in \{0,1\}^{n^r} \text{ and } P\nprec A\}$.  
We can view $P$ as an ordered hypergraph $H(P)$ with interval chromatic number $\chi_<(H(P))=r$ by ordering the 
$k_1+\cdots+k_r$ vertices 
primarily by their axis, then according to their coordinate within the axis. 
Each $P(i_1,\ldots,i_r)=1$ 
corresponds to a 
hyperedge $\{i_1^{(1)},\ldots,i_r^{(r)}\}$, where $i_j^{(j)}$ is the vertex in axis $j$ and position $i_j$.

\begin{lemma}[{Cf.~\cite[Thm.~2]{PachTardos06}}]\label{lem:PT-multidimensional}
Suppose $P\in\{0,1\}^{k_1\times\cdots\times k_r}$ and $H(P)$ has no isolated vertices.
Then 
$\Ex^r(P,n) \leq \ExH^r(H(P),rn)$.
\end{lemma}

\begin{proof}
If $A\in\{0,1\}^{r^n}$ is $P$-free,
then $H(A)$ is an ordered $r$-uniform 
hypergraph on $rn$ vertices that is $H(P)$-free,
as every embedding of $H(P)$ in $H(A)$
must assign the $r$ vertices of each hyperedge
into vertices associated with 
distinct axes of $A$.
\end{proof}

\medskip 

Recall that \cref{conj:PachTardos} is equivalent to the $r=2$ case of \cref{conj:Furedietal-ordered-hypergraph}. Thus, \cref{thm:PT-refutation} refuting the former gives a counterexample of the latter conjecture that is a $2$-uniform ordered hypergraph, i.e., a vertex-ordered graph. We thought that providing counterexamples that are ``real'' hypergraphs is valuable, so here we give a similar counterexample that is an $r$-uniform ordered hypergraphs for any $r>2$, as claimed in the first paragraph of Section~\ref{sec:newresults}.

Let $S_0^r \in \{0,1\}^{3\times 4\times 1\times \cdots \times 1}$
be the $r$-dimensional version of the pattern $S_0$, 
where the last $r-2$ dimensions have width 1.
Then $H(S_0^r)$ is in fact an ordered $r$-uniform 
forest hypergraph with $r+5$ vertices.
Construct a matrix $A^r \in \{0,1\}^{n^r}$ such that 
every square submatrix 
$A \in \{0,1\}^{n\times n}$ 
obtained by fixing a single coordinate in the last $r-2$ axes 
is a copy of the lower bound construction $A$ from 
\cref{thm:PT-refutation}.
Then $\Ex^r(S_0^r,n) \geq n^{r-2} \|A\|_1 = n^{r-1}2^{\sqrt{\log n}-O(\log\log n)}$.  By \cref{lem:PT-multidimensional},
$\ExH^r(H(S_0^r),n) \geq \Ex^r(S_0^r,n/r) = \Omega(n^{r-1}2^{\sqrt{\log n}-O(\log\log n)})$, which refutes \cref{conj:Furedietal-ordered-hypergraph} for all $r\ge 2$ as $H(S_0^r)$ is an $r$-uniform forest of interval chromatic number $r$.

\end{document}